\documentclass[article]{elsarticle}
\usepackage{lineno,hyperref}
\modulolinenumbers[5]
\usepackage{amsmath,amssymb,amsthm}
\usepackage{color}
\usepackage{stmaryrd}
\usepackage{graphicx}
\usepackage{lineno,hyperref}
\usepackage{subcaption}
\modulolinenumbers[5]
\usepackage{mathtools}
\usepackage{ifthen}
\usepackage{enumitem}
\usepackage{tikz}
\usepackage{graphicx,tcolorbox}
\usepackage{rotating}
\usepackage{color}
\usepackage{xcolor}
\newtheorem{remark}{Remark}

\newtheorem{thm}{Theorem}

\newtheorem{lem}{Lemma}

\renewcommand{\div}{{\textrm{div }}}
\newcommand{\bdev}{{\bf \text{dev }}}

\newcommand{\bzero}{{\bf 0}}

\newcommand{\norm}[1]{\lVert #1\rVert}
\newcommand{\bI}{{\bf I}}

\newcommand{\R}{\mathbb R}

\newcommand{\bJ}{\mathbf J}

\newcommand{\cS}{{\mathcal S}}
\newcommand{\bP}{{\mathbf P}}
\newcommand{\bQ}{{\mathbf Q}}
\newcommand{\bZ}{{\mathbf Z}}

\newcommand{\bn}{{\bf n}}
\newcommand{\bW}{{\bf W}}

\newcommand{\cT}{\mathcal T}

\newcommand{\bepsilon}{\mbox{\boldmath$\varepsilon$}}

\newcommand{\cV}{\mathcal V}

\newcommand{\Disp}{{\boldsymbol{u}}}
\newcommand{\bzeta}{{\boldsymbol{\zeta}}}
\newcommand{\bu}{{\boldsymbol{u}}}
\newcommand{\bR}{{\boldsymbol{R}}}
\newcommand{\bM}{{\boldsymbol{M}}}

\newcommand{\Stress}{{\boldsymbol{\sigma}}}

\newcommand{\TotalStressdef}[2]{
		\ifthenelse{\equal{#1}{}}
		{\Stress(\Disp) - \alpha \phi  \bI}
		{\Stress(#1) - \alpha #2  \bI}
}
\newcommand{\bTheta}{{\boldsymbol{\Theta}}}
\newcommand{\bgamma}{{\boldsymbol{\gamma}}}
\newcommand{\bz}{{\boldsymbol{z}}}
\newcommand{\bw}{{\boldsymbol{w}}}
\newcommand{\brho}{{\boldsymbol{\rho}}}

\newcommand{\mynote}[1]{}
\newcommand{\bff}{\boldsymbol{f}}

\newcommand{\symgrad}[2]{
\ifthenelse{\equal{#2}{}}
	{
		\ifthenelse{\equal{#1}{}}
		{\boldsymbol\varepsilon}
		{\boldsymbol\varepsilon(#1)}
	}
	{(\boldsymbol\varepsilon(#1),\boldsymbol\varepsilon(#2))}
}
\newcommand{\energy}[1]{
{{\left\vert\kern-0.25ex\left\vert\kern-0.25ex\left\vert #1 
    \right\vert\kern-0.25ex\right\vert\kern-0.25ex\right\vert}}
}

\newcommand{\atwo}[4]{
\ifthenelse{\equal{#2}{}}
	{
 (s_0 #3,#4)+(\alpha \div #1 , #4)+(\kappa \nabla #3,\nabla #4)
	}
	{
\ifthenelse{\equal{#4}{}}
	{
(\symgrad{#1}{#2})+(\lambda(\div #1) \bI
- \alpha #3  \bI,\nabla #2)
	}
	{
(\symgrad{#1}{#2})+(\lambda(\div #1) \bI
- \alpha #3  \bI,\nabla #2)
 +
 (s_0 #3,#4)+(\alpha \div #1 , #4)+(\kappa \nabla #3,\nabla #4)
 }
 }
}

\newcommand{\bV}{{\bf V}}
\newcommand{\bv}{{\bf v}}

\newcommand{\btheta}{\mbox{\boldmath$\theta$}}

\newcommand{\energyds}[3]{
\ifthenelse{\equal{#1}{2}}
	{
\norm{\symgrad{#2}{}}^2+\kappa\norm{\nabla #3}^2
+s_0\norm{ #3}^2
	}
	{
	
	}
}

\newcommand{\tr}{{\rm tr}}

\newcommand{\cA}{{\cal A}}
\newcommand{\bxi}{\mbox{\boldmath$\xi$}}
\newcommand{\bas}{{\bf as}}

\newcommand{\bnabla}{\mbox{\boldmath$\nabla$}}
\begin{document}

\title{A posteriori error estimates by weakly symmetric stress reconstruction for the Biot problem}

\author[1]{Fleurianne Bertrand}
\author[2]{Gerhard Starke}

\address[1]{{%
Institut f\"ur Mathematik}, {Humboldt-Universit\"at zu Berlin}, {Unter den Linden 6, 10099 Berlin, {Germany}.\\
{fb@math.hu-berlin.de}}}
        
\address[2]{{%
Fakult\"at f\"ur Mathematik}, {Universit\"at Duisburg-Essen}, {Thea-Leymann-Str. 9, 45127 Essen, {Germany}.\\
{gerhard.starke@uni-due.de}}}

\begin{abstract}
A posteriori error estimates are constructed for the three-field variational formulation of the Biot problem involving the displacements, the total pressure and
the fluid pressure. The discretization under focus is the $H^1(\Omega)$-conforming Taylor-Hood finite element combination, consisting of polynomial
degrees $k + 1$ for the displacements and the fluid pressure and $k$ for the total pressure. An a posteriori error estimator is derived on the basis of
$H (\textrm{div})$-conforming reconstructions of the stress and flux approximations. The symmetry of the reconstructed stress is allowed to be satisfied only weakly.
The reconstructions can be performed locally
on a set of vertex patches and lead to a guaranteed upper bound for the error with a constant that depends only on local constants associated with the
patches and thus on the shape regularity of the triangulation. Particular emphasis is given to nearly incompressible materials and the error estimates hold
uniformly in the incompressible limit.
% and are robust with respect to the Biot parameter.
Numerical results on the L-shaped domain confirm the theory and the suitable use of the error estimator in adaptive strategies.
\end{abstract} 

\maketitle

\section*{Acknowledgement}
The authors gratefully acknowledge support by the Deutsche Forschungsgemeinschaft in the Priority Program SPP 1748 ‘Reliable simulation techniques in solid mechanics. Development of non standard discretization methods, mechanical and mathematical analysis’ under the project numbers BE 6511/1-1 and STA 402/12-2

\section{Introduction}

The classical Biot model of consolidation for poroelastic media presented in \cite{Bio:41} describes a linearly elastic and porous structure, which is saturated
by a slightly compressible viscous fluid. This system of partial differential equations has a wide range of applications, in particular in geo- and biomechanics.
Their numerical treatment is therefore of crucial importance. As a coupled problem, possibly involving different scales, the derivation of reliable and
efficient error estimates is a challenging task.

In the classical two-fields formulation investigated in \cite{MurThoLou:96}, the system is given by in the conservation of mechanical momentum and the
conservation of the fluid mass, expressed with the displacement field $\Disp$  and the fluid pressure $\varphi$. In particular, it was shown that the
$H^1$-conforming Taylor–Hood finite element combination leads to a optimal a priori estimates,
see also the references in \cite{MurThoLou:96} about the origin of this approach to the Biot model.
A major drawback of this two-fields formulation is the lack
of robustness with respect to the incompressibility parameter. Therefore, several three-fields formulations with various discretisations were proposed, e.g. in
\cite{PhiWhe:08, OyaRui:16, LeeMarWin:17} with particular emphasis given to the robustness with respect to model parameters of the Biot’s model.
Our focus is on the solution of the elliptic system arising from one time-step of an implicit Euler discretization in time. The a posteriori
treatment of space and time error is carried out in \cite{ErnMeu:09}, for example.

Our a posteriori error estimator will be based on the reconstruction of flux and stress in suitable finite element spaces. There is obviously
a strong connection to variational formulations of the Biot problem where flux and stress are already included as independent variables (cf.
\cite{KorSta:05}, \cite{Lee:16}).
A posteriori estimates for the error associated with the two-field formulation by flux and stress reconstruction suitable for an adaptive strategy were
developed in \cite{RieDipErn:17}.
The idea can be traced back to \cite{PraSyn:47} and has led to a posteriori error bounds already in pioneer work \cite{LadLeg:83}. A practical algorithm
based on the idea of equilibration in broken Raviart-Thomas spaces was
presented in \cite{BraSch:08,BraPilSch:09}. A unified framework for a posteriori bounds by flux reconstructions for the finite element approximation of elliptic
problem restricted to bilinearforms involving the full gradient can be find in \cite{ErnVoh:15}. The extension to the two-field formulation of the Biot problem
involves the reconstruction of the flux of the fluid pressure using the above references as well as the reconstruction of total stress tensor. In this tensor
reconstruction, the anti-symmetric part has to be controlled for the use in an associated a posteriori error estimator. Surely, the reconstruction can be
performed in a symmetric space like the Arnold-Winther space from \cite{ArnWin:02} but  a weakly symmetric reconstruction in the Raviart-Thomas finite
element space as in \cite{BerKobMolSta:19} is preferable due to its less complicated structure. The main result of our paper is the a posteriori
error bound obtained by a weakly symmetric reconstruction of the total stress tensor combined with a reconstruction of the Darcy velocity.

The paper is organized as follows. The next section reviews the three-field formulation for the Biot problem and introduce a suitable discretization. In section 3,  the a posteriori error estimator is derived together with the conditions for a weakly symmetric stress equilibration. The equilibration procedure to satisfy these conditions is presented in section 4. Finally, section 5 shows numerical results both for a regular and a singular problem. 

\section{A three-field formulation of the Biot system}

We consider the following three-field formulation that was introduced in \cite[Sect. 3.2]{LeeMarWin:17}: Find $\bu \in H_0^1 (\Omega)^d$,
$p \in L^2 (\Omega)$ and $\varphi \in H_0^1 (\Omega)$ such that
\begin{equation}
  \begin{split}
    2 \mu ( \bepsilon (\bu) , \bepsilon (\bv) ) - ( p , \div \: \bv ) & = ( \bff , \bv ) \hspace{1cm} \forall \bv \in H_0^1 (\Omega)^d \: , \\
    ( \div \: \bu , q ) + \frac{1}{\lambda} ( p - \varphi , q ) & = 0 \hspace{1.5cm} \forall q \in L^2 (\Omega) \: , \\
    \frac{1}{\lambda} ( \varphi - p , \psi ) + \tau ( \nabla \varphi , \nabla \psi ) & = ( g , \psi ) \hspace{1cm} \forall \psi \in H_0^1 (\Omega) \: .
  \end{split}
  \label{eq:Biot_system}
\end{equation}
Here and throughout our paper, $( \: \cdot \: , \: \cdot \: )$ denotes the $L^2 (\Omega)$ inner product.
This coincides with the system \cite[(3.13)]{LeeMarWin:17} if we set $\alpha = 1$, $\kappa = \tau$ and ignore the storage coefficient
$s_0 \sim 1/\lambda$ there (and set $\mu = 1/2$ in (\ref{eq:Biot_system})).
It arises from an implicit Euler time discretization of the parabolic problem associated with the Biot model.
Concerning the robustness of the formulation
with respect to the parameters $\alpha$ and $s_0$, the discussion in \cite[Sect. 3]{LeeMarWin:17} applies.
Note that (\ref{eq:Biot_system}) decouples into an
incompressible elasticity system and a Poisson problem as $\lambda \rightarrow \infty$.
The discretized system consists in finding $\bu_h \in \bV_h$, $p_h \in Q_h$ and $\varphi_h \in S_h$ such that
\begin{equation}
  \begin{split}
    2 \mu ( \bepsilon (\bu_h) , \bepsilon (\bv_h) ) - ( p_h , \div \: \bv_h ) & = ( \bff , \bv_h ) \hspace{1cm} \forall \bv \in \bV_h \: , \\
    ( \div \: \bu_h , q_h ) + \frac{1}{\lambda} ( p_h - \varphi_h , q_h ) & = 0 \hspace{1.5cm} \forall q_h \in Q_h \: , \\
    \frac{1}{\lambda} ( \varphi_h - p_h , \psi_h ) + \tau ( \nabla \varphi_h , \nabla \psi_h ) & = ( g , \psi_h ) \hspace{1cm} \forall \psi_h \in S_h \: ,
  \end{split}
  \label{eq:Biot_system_discrete}
\end{equation}
where $\bV_h \subset H_0^1 (\Omega)^d$, $Q_h \subseteq L^2 (\Omega)$ and $S_h \subset H_0^1 (\Omega)$ are suitable subspaces.
In particular, $V_h$ and $Q_h$ need to satisfy the inf-sup condition with respect to the divergence constraint. Moreover, it is desirable that
all the terms in the energy norm
\begin{equation}
  \begin{split}
    ||| & ( \bu - \bu_h , p - p_h , \varphi - \varphi_h ) ||| \\
    & = \left( 2 \mu \| \bepsilon (\bu - \bu_h) \|^2 + \frac{1}{\lambda} \| p - p_h - (\varphi - \varphi_h) \|^2
    + \tau \| \nabla (\varphi - \varphi_h) \|^2 \right)^{1/2}
  \end{split}
  \label{eq:energy_norm}
\end{equation}
are convergent of the same order. (Note that this is a norm on $H_0^1 (\Omega)^d \times L^2 (\Omega) \times H_0^1 (\Omega)$
since $\nabla (\varphi - \varphi_h) = 0$ implies $\varphi - \varphi_h = 0$ and therefore also $p - p_h = 0$.)
To this end, we combine Taylor-Hood elements (conforming piecewise quadratic for $\bV_h$ with
conforming piecewise linear for $Q_h$ in the lowest-order case, see \cite[Sect. 8.8]{BofBreFor:13}) with conforming piecewise quadratic
elements for $S_h$.

The stress and flux, associated with (\ref{eq:Biot_system}) and (\ref{eq:Biot_system_discrete}), are given by
\begin{equation}
  \begin{split}
    \btheta (\bu,p,\varphi) = 2 \mu \bepsilon (\bu) - ( p - \varphi ) \bI \: , \: \hspace{0.2cm} &
    \btheta (\bu_h,p_h,\varphi_h) = 2 \mu \bepsilon (\bu_h) - ( p_h - \varphi_h ) \bI \: , \: \\
    \bw (\varphi) = - \nabla \varphi \: , \: \hspace{2cm} & \;\;\;\; \bw (\varphi_h) = - \nabla \varphi_h \: ,
  \end{split}
  \label{eq:stress_flux}
\end{equation}
respectively. Note that the definition of the stress in (\ref{eq:stress_flux}) implies
\begin{equation}
  \tr \: \btheta (\bu,p,\varphi) = 2 \mu \: \div \: \bu - d \: ( p - \varphi ) \: .
  \label{eq:trace_stress}
\end{equation}
Our error estimator will be based on reconstructions of the stress and flux which can be computed as a correction to
$\btheta (\bu_h,p_h,\varphi_h)$ and $\bw (\varphi_h)$, respectively. This will be explained in detail in section \ref{sec-reconstruction}.
Motivated by the fact that
$\div \: \btheta (\bu,p,\varphi) = - \bff + \nabla \varphi$ holds, the stress reconstruction $\btheta_h^R$ is required to satisfy
\begin{equation}
  \div \: \btheta_h^R = - \bff + \nabla \varphi_h \mbox{ in } \Omega
  \label{eq:momentum_balance}
\end{equation}
(assuming that $\bff$ is contained in the corresponding space of piecewise polynomials, note that $\nabla S_h \subset \bTheta_h^R$
holds for the Raviart-Thomas elements used for reconstruction in section \ref{sec-reconstruction}).

Similarly, the identity
$\tau \: \div \: \bw (\varphi) = g + (p - \varphi)/\lambda$ of the exact flux motivates the condition
\begin{equation}
  \tau \: \div \: \bw_h^R = g + \frac{p_h - P_h^1 \varphi_h}{\lambda} \mbox{ in } \Omega
  \label{eq:mass_balance}
\end{equation}
for the flux reconstruction. Here, $P_h^1$ denotes the $L^2$-orthogonal projection onto the space of piecewise linear (possibly discontinuous)
functions which reflects the fact that $w_h^R$ (as well as $\btheta_h^R$) will be constructed in the Raviart-Thomas space (of next-to-lowest order).
For more general right-hand sides $\widetilde{\bff}$ and $\widetilde{g}$, the system (\ref{eq:Biot_system}) is considered with
$\bff = \bP_h^1 \widetilde{\bff}$ (the vector version of the $L^2$-orthogonal projection) and $g = P_h^1 \widetilde{g}$. The associated error
$( \widetilde{\bu} - \bu , \widetilde{p} - p , \widetilde{\varphi} - \varphi )$ can then be bounded by the approximation errors
$\widetilde{\bff} - \bP_h^1 \widetilde{\bff}$ and $\widetilde{g} - P_h^1 \widetilde{g}$ due to the stability proved in \cite{LeeMarWin:17}.

\section{A posteriori error estimation by stress and flux reconstruction}

Our error estimator, based on the stress and flux reconstructions described in the previous section, will be based on the quantities
\begin{equation}
  \begin{split}
    \eta_S & = \| \btheta_h^R - \btheta (\bu_h,p_h,\varphi_h) \|_\cA \\
    & = ( \btheta_h^R - \btheta (\bu_h,p_h,\varphi_h) , \cA (\btheta_h^R - \btheta (\bu_h,p_h,\varphi_h)) )^{1/2} \: ,
  \end{split}
  \label{eq:estimator_term_stress}
\end{equation}
where $\cA : \R^{d \times d} \rightarrow \R^{d \times d}$ is defined by
\begin{equation}
  \cA \bxi = \frac{1}{2 \mu} \left( \bxi - \frac{\lambda}{2 \mu + d \lambda} (\tr \: \bxi) \bI \right) \: ,
  \label{eq:strain_stress}
\end{equation}
and
\begin{equation}
  \eta_F = \tau^{1/2} \| \bw_h^R - \bw (\varphi_h) \| \;\; , \;\; \eta_P = \frac{1}{{\color{black} \lambda} \tau^{1/2}} \| \varphi_h - P_h^1 \varphi_h \| \: .
  \label{eq:estimator_term_flux}
\end{equation}
The idea for using the additional term $\eta_P$ for the consideration of the lower order term of the third equation in (\ref{eq:Biot_system}) is taken
from \cite{AinVej:19}. Moreover, the additional terms
\begin{equation}
    \eta_A = \| \bas \: \btheta_h^R \| \: , \:
%    \eta_T & = \frac{1}{\lambda^{1/2}} \| \tr \: \btheta_h^R - \frac{2 \mu}{\lambda} \varphi_h + \frac{2 \mu + d \lambda}{\lambda} p_h \| \: , \\
    \eta_C = \| \frac{p_h - \varphi_h}{\lambda} + \div \: \bu_h \|
  \label{eq:estimator_term_deformation}
\end{equation}
will be needed for the error estimation of the deformation problem.

We first consider the stress estimator contribution $\eta_S$ in order to derive an upper bound for the error associated with the deformation problem.

\begin{lem}
  For the contribution of the stress reconstruction to the error estimator, we have
  \begin{equation}
    \begin{split}
      \mu \| \bepsilon (\bu - \bu_h) \|^2 + & \frac{1}{\lambda} \| p - p_h - (\varphi - \varphi_h) \|^2 
      - 2 ( \varphi - \varphi_h , {\rm div} (\bu - \bu_h) ) \\
      & \leq \eta_S^2 + \frac{C_K^2}{4 \mu} \eta_A^2 + {\color{black} \mu} \left( \frac{\lambda C_D}{2 \mu + d \lambda} + \frac{1}{C_D} \right)^2 \eta_C^2 \: .
    \end{split}
    \label{eq:stress_bound}
  \end{equation}
  \label{lem-stress_bound}
\end{lem}

\begin{proof}
  Adding and substracting the exact stress $\btheta (\bu,p,\varphi) = 2 \mu \bepsilon (\bu) - (p - \varphi) \: \bI$ leads to
  \begin{equation}
    \begin{split}
      \eta_S^2 & = \| \btheta_h^R - \btheta (\bu,p,\varphi) + \btheta (\bu,p,\varphi) - \btheta (\bu_h,p_h,\varphi_h) \|_\cA^2 \\
      & = \| \btheta_h^R - \btheta (\bu,p,\varphi) \|_\cA^2 \\
      & + ( 2 (\btheta_h^R - \btheta (\bu,p,\varphi)) + \btheta(\bu,p,\varphi) - \btheta(\bu_h,p_h,\varphi_h) , \\
      & \hspace{4cm} \cA (\btheta(\bu,p,\varphi)-\btheta(\bu_h,p_h,\varphi_h)) ) \: .
    \end{split}
    \label{eq:stress_1}
  \end{equation}
  Inserting
  \begin{equation}
    \begin{split}
       \cA ( & \btheta (\bu,p,\varphi) - \btheta(\bu_h,p_h,\varphi_h) ) \\
      & = \cA (2 \mu \bepsilon (\bu - \bu_h) - ((p-p_h) - (\varphi - \varphi_h)) \bI) ) \\
      & = \bepsilon (\bu-\bu_h)
      - \frac{\lambda}{2 \mu + d \lambda} \left( \div (\bu - \bu_h) + \frac{p - p_h - (\varphi - \varphi_h)}{\lambda} \right) \bI \: , \\
      & = \bepsilon (\bu-\bu_h)
      + \frac{\lambda}{2 \mu + d \lambda} \left( \div \bu_h + \frac{p_h - \varphi_h}{\lambda} \right) \bI \: ,
%      & = \bepsilon (\bu-\bu_h) - \frac{\lambda}{2 \mu + d \lambda} \left( \frac{\varphi - p_h}{\lambda} - \div \: \bu_h \right) \bI \\
%      & = \bepsilon (\bu-\bu_h) - \frac{\lambda}{2 \mu + d \lambda}
%      \left( \frac{\varphi - \varphi_h}{\lambda} - ( \frac{p_h - \varphi_h}{\lambda} + \div \: \bu_h ) \right) \bI \: , \\
    \end{split}
    \label{eq:stress_2}
  \end{equation}
  which follows using the definition (\ref{eq:strain_stress}) and the second equation in (\ref{eq:Biot_system}), implies
  \begin{equation}
    \begin{split}
      \eta_S^2 = & \| \btheta_h^R - \btheta (\bu,p,\varphi) \|_\cA^2 + 2 (\btheta_h^R - \btheta (\bu,p,\varphi) , \bepsilon (\bu-\bu_h) ) \hspace{2cm} \\
      & + \frac{2 \lambda}{2 \mu + d \lambda} ( \tr \: (\btheta_h^R - \btheta (\bu,p,{\color{black} \varphi})) , \div \: \bu_h + \frac{p_h - \varphi_h}{\lambda} ) \\
      & + 2 \mu \| \bepsilon (\bu - \bu_h) \|^2 + \frac{1}{\lambda} \| p - p_h - (\varphi - \varphi_h) \|^2 \\
      & - \frac{2 \mu \lambda}{2 \mu + d \lambda} \| \div \bu_h + \frac{p_h - \varphi_h}{\lambda} \|^2 \: .
    \end{split}
    \label{eq:stress_3}
  \end{equation}
  Integration by parts leads to
  \begin{equation}
    \begin{split}
      (\btheta_h^R - & \btheta (\bu,p,\varphi) , \bepsilon (\bu-\bu_h) ) \\
      & = (\btheta_h^R - \btheta (\bu,p,\varphi) , \bnabla (\bu-\bu_h) - \bas \: \bnabla (\bu-\bu_h) ) \\
      & = ( \div \btheta (u,p,\varphi) - \div \btheta_h^R , \bu - \bu_h ) - (\bas \: \btheta_h^R , \bas \: \bnabla (\bu-\bu_h) ) \\
      & = ( \nabla (\varphi - \varphi_h) , \bu - \bu_h ) - (\bas \: \btheta_h^R , \bas \: \bnabla (\bu-\bu_h) ) \\
      & = - ( \varphi - \varphi_h , \div (\bu - \bu_h) ) - (\bas \: \btheta_h^R , \bas \: \bnabla (\bu-\bu_h) ) \: ,
    \end{split}
    \label{eq:stress_integration_by_parts}
  \end{equation}
  since, by construction, $\div \: \btheta (\bu,p,\varphi) - \div \: \btheta_h^R = \nabla (\varphi - \varphi_h)$ in $\Omega$ and $\bu_h = \bu = \bzero$ on
  $\partial \Omega$. Therefore, (\ref{eq:stress_3}) can be rewritten as
  \begin{equation}
    \begin{split}
      \eta_S^2 = & \| \btheta_h^R - \btheta (\bu,p,\varphi) \|_\cA^2 \\
      & - 2 ( \varphi - \varphi_h , \div (\bu-\bu_h) ) - ( \bas \: \btheta_h^R , \bas \: \bnabla (\bu-\bu_h) ) \hspace{2cm} \\
      & + \frac{2 \lambda}{2 \mu + d \lambda} ( \tr \: (\btheta_h^R - \btheta (\bu,p,{\color{violet} \varphi}))) , \div \: \bu_h + \frac{p_h - \varphi_h}{\lambda} ) \\
      & + 2 \mu \| \bepsilon (\bu - \bu_h) \|^2 + \frac{1}{\lambda} \| p - p_h - (\varphi - \varphi_h) \|^2 \\
      & - \frac{2 \mu \lambda}{2 \mu + d \lambda} \| \div \bu_h + \frac{p_h - \varphi_h}{\lambda} \|^2 \: .
    \end{split}
    \label{eq:stress_4}
  \end{equation}
  Using the inequalities
  \begin{equation}
    \begin{split}
      | ( \bas \: \btheta_h^R , \bas \: \nabla (\bu - \bu_h) ) | & \leq C_K \| \bas \: \btheta_h^R \| \: \| \bas \: \nabla (\bu - \bu_h) ) \| \: , \\
      | ( \tr (\btheta_h^R - \btheta (\bu,p,{\color{black} \varphi}))) , & \div \bu_h + \frac{p_h - \varphi_h}{\lambda} ) | \\
      \leq C_D & {\color{black} \mu^{1/2}} \| \btheta_h^R - \btheta (\bu,p,{\color{black} \varphi})) \|_\cA
      \| \frac{p_h - \varphi_h}{\lambda} + \div \: \bu_h \|
    \end{split}
    \label{eq:local_bounds_both}
  \end{equation}
  which will be proved in section \ref{sec-reconstruction} as Theorem \ref{thm-local_bounds}, we are led from (\ref{eq:stress_4}) to
  \begin{align*}
      \eta_S^2 & \geq \| \btheta_h^R - \btheta (\bu,p,\varphi) \|_\cA^2
      - 2 ( \varphi - \varphi_h , \div (\bu-\bu_h) ) - C_K \| \bas \: \btheta_h^R \| \: \| \bepsilon (\bu-\bu_h) \| \hspace{2cm} \\
      & - \frac{2 \lambda \mu^{1/2}}{2 \mu + d \lambda} C_D \| \btheta_h^R - \btheta (\bu,p,{\color{black} \varphi}))) \|_\cA \:
      \| \div \: \bu_h + \frac{p_h - \varphi_h}{\lambda} \| \\
      & + 2 \mu \| \bepsilon (\bu - \bu_h) \|^2 + \frac{1}{\lambda} \| p - p_h - (\varphi - \varphi_h) \|^2
      - \frac{2 \mu \lambda}{2 \mu + d \lambda} \| \div \bu_h + \frac{p_h - \varphi_h}{\lambda} \|^2
  \end{align*}
  and from this, using the Young inequality, to
  \begin{equation}
    \begin{split}
      \mu & \| \bepsilon (\bu - \bu_h) \|^2 + \frac{1}{\lambda} \| p - p_h - (\varphi - \varphi_h) \|^2 
      - 2 ( \varphi - \varphi_h , \div (\bu - \bu_h) ) \\
      \leq & \eta_S^2 + \frac{C_K^2}{4 \mu} \| \bas \btheta_h^R \|^2
      + \left( \mu ( \frac{\lambda C_D}{2 \mu + d \lambda} )^2 + \frac{2 \lambda}{2 \mu + d \lambda} \right)
      \| \div \bu_h + \frac{p_h - \varphi_h}{\lambda} \|^2 \: .
    \end{split}
    \label{eq:stress_5}
  \end{equation}
  Finally, this proves our estimate (\ref{eq:stress_bound}).
  \end{proof}
  
\begin{lem}
  For the contribution of the flux reconstruction to the error estimation, we have
  \begin{equation}
    \frac{3}{4} \tau \| \nabla (\varphi - \varphi_h) \|^2 + \frac{2}{\lambda} ( \varphi - \varphi_h - (p - p_h) , \varphi - \varphi_h )
    \leq \eta_F^2 + 4 C_F^2 \eta_P^2 \: .
    \label{eq:flux_bound}
  \end{equation}
  \label{lem-flux_bound}
\end{lem}

\begin{proof}
  Adding and substracting the exact flux $\bw (\varphi) = - \nabla \varphi$ leads to
  \begin{equation}
    \begin{split}
      \eta_F^2 & = \tau \| \bw_h^R - \bw (\varphi) + \bw (\varphi) - \bw (\varphi_h) \|^2
      = \tau \| \bw_h^R - \bw (\varphi) - \nabla (\varphi - \varphi_h) \|^2 \\
      & = \tau \| \bw_h^R - \bw (\varphi) \|^2 - 2 \tau (\bw_h^R - \bw (\varphi) , \nabla (\varphi - \varphi_h) ) + \tau \| \nabla (\varphi - \varphi_h) \|^2 \\
      & \geq \frac{2}{\lambda} ( \varphi - P_h^1 \varphi_h - (p - p_h) , \varphi - \varphi_h ) + \tau \| \nabla (\varphi - \varphi_h) \|^2 \: ,
    \end{split}
    \label{eq:flux_1}
  \end{equation}
  where, after integrating by parts, the conservation properties of the reconstructed and of the exact flux,
  \[
    \tau \: \div \: \bw_h^R = g+ \frac{p_h - P_h^1 \varphi_h}{\lambda} \mbox{ and } \tau \: \div \: \bw (\varphi)
    = g + \frac{p - \varphi}{\lambda} \: ,
  \]
  respectively, were used. This can be rewritten as
  \begin{equation}
    \begin{split}
      \eta_F^2 & \geq \frac{2}{\lambda} ( \varphi - \varphi_h - ( p - p_h ) , \varphi - \varphi_h )
      + \frac{2}{\lambda} ( \varphi_h - P_h^1 \varphi_h , \varphi - \varphi_h ) \\
      & \hspace{7cm} + \tau \| \nabla (\varphi - \varphi_h) \|^2 \: .
    \end{split}
    \label{eq:flux_2}
  \end{equation}
  {\color{black} The middle term on the right-hand side in (\ref{eq:flux_2}) may be bounded, using the Young inequality, in the form
  \begin{equation}
    \frac{2}{\lambda} ( \varphi_h - P_h^1 \varphi_h , \varphi - \varphi_h ) \geq 
    - \frac{4 C_F^2}{\tau \lambda^2} \| \varphi_h - P_h^1 \varphi_h \|^2 - \frac{\tau}{4 C_F^2} \| \varphi - \varphi_h \|^2 \: .
    \label{eq:flux_intermediate}
  \end{equation}
  } 
  Using the Poincar\'e-Friedrichs inequality in the form
  \begin{equation}
    \| \varphi - \varphi_h \| \leq C_F \| \nabla (\varphi - \varphi_h) \| \: ,
    \label{eq:PF}
  \end{equation}
  \begin{equation}
    \begin{split}
      \frac{3}{4} \tau \| \nabla (\varphi - \varphi_h) \|^2 + \frac{2}{\lambda} ( \varphi - \varphi_h - (p - p_h) , & \varphi - \varphi_h ) \\
      & \leq \eta_F^2 + 
      %\frac{4 C_F^2}{\lambda \tau} i changed that constant here!
      \frac{4 C_F}{{\color{black} \lambda^2} \tau} \| \varphi_h - P_h^1 \varphi_h \|^2
    \end{split}
    \label{eq:flux_3}
  \end{equation}
  is obtained which finally proves (\ref{eq:flux_bound}).
\end{proof}

Combining Lemma \ref{lem-stress_bound} and Lemma \ref{lem-flux_bound}, we obtain an upper bound for the total error associated with the
Biot system.

\begin{thm}
  The total error in the energy norm associated with the Biot system satisfies
  \begin{equation}
    \begin{split}
      & 2 \mu \| \bepsilon (\bu - \bu_h) \|^2 
      + \frac{1}{\lambda} \| p - p_h {\color{black} -(\varphi-\varphi_h)} \|^2 + \tau \| \nabla (\varphi - \varphi_h) \|^2 \\
      & \leq 2 \left( \eta_S^2 + \frac{C_K^2}{4 \mu} \eta_A^2
      + \left( {\color{black} \mu} ( \frac{C_D \lambda}{2 \mu + d \lambda} + \frac{1}{C_D} )^2 + \frac{4 C_F^2}{\tau} \right)
      \eta_C^2 + \eta_F^2 + 4 C_F^2 \eta_P^2 \right) .
    \end{split}
    \label{eq:total_bound}
  \end{equation}  
  \label{thm-total_bound}
\end{thm}

\begin{proof}
  Adding (\ref{eq:stress_bound}) and (\ref{eq:flux_bound}) gives
  \begin{equation}
    \begin{split}
      \mu \| \bepsilon (\bu & - \bu_h) \|^2 
      + \frac{1}{\lambda} 
      \| p - p_h - (\varphi - \varphi_h) \|^2 
      + \frac{3}{4} \tau \| \nabla (\varphi - \varphi_h) \|^2 \\
      & \hspace{2cm} + 2 ( \varphi - \varphi_h , \div \bu_h + \frac{p_h - \varphi_h}{\lambda} ) \\
      & \leq \eta_S^2 + \frac{C_K^2}{4 \mu} \eta_A^2 + {\color{black} \mu} \left( \frac{\lambda C_D}{2 \mu + d \lambda} + \frac{1}{C_D} \right)^2 \eta_C^2
      + \eta_F^2 + 4 C_F^2 \eta_P^2 \: ,
    \end{split}
    \label{eq:total_1}
  \end{equation}
  where the second equation of (\ref{eq:Biot_system}) was used again. With the Poincar\'e-Friedrichs inequality (\ref{eq:PF}) and
  \[
    2 ( \varphi - \varphi_h , \div \bu_h + \frac{p_h - \varphi_h}{\lambda} )
    \geq - \frac{\tau}{4 C_F^2} \| \varphi - \varphi_h \|^2 - \frac{4 C_F^2}{\tau} \| \div \bu_h + \frac{p_h - \varphi_h}{\lambda} \|^2
  \]
  this leads to
  \begin{equation}
    \begin{split}
      \mu & \| \bepsilon (\bu - \bu_h) \|^2 + \frac{1}{\lambda} \| p - p_h - (\varphi - \varphi_h) \|^2 + \frac{1}{2} \tau \| \nabla (\varphi - \varphi_h) \|^2 \\
      & \leq \eta_S^2 + \frac{C_K^2}{4 \mu} \eta_A^2
      + \left( {\color{black} \mu} \left( \frac{\lambda C_D}{2 \mu + d \lambda} + \frac{1}{C_D} \right)^2 + \frac{4 C_F^2}{\tau} \right) \eta_C^2
      + \eta_F^2 + 4 C_F^2 \eta_P^2
    \end{split}
    \label{eq:total_2}
  \end{equation}
  which finishes the proof.
\end{proof}

\section{Reconstruction of stress and flux}

\label{sec-reconstruction}

Our reconstructed stress $\btheta_h^R \in \bTheta_h^R$ and flux $\bw_h^R \in \bW_h^R$ will be constructed in the spaces
$\bTheta_h^R \subset H (\div,\Omega)^d$ and $\bW_h^R \subset H (\div,\Omega)$ given by the next-to-lowest order Raviart-Thomas spaces
$RT_1^d$ and $RT_1$, respectively (see \cite[Sect. 2.3.1]{BofBreFor:13}).
They may be computed as a correction $\btheta_h^R = \btheta (\bu_h,p_h,\varphi_h) + \btheta_h^\Delta$ and
$\bw_h^R = \bw (\varphi_h) + \bw_h^\Delta$,
respectively. The corrections $\btheta_h^\Delta$ and $\bw_h^\Delta$ are contained in the broken Raviart-Thomas space
\begin{equation}
  \begin{split}
    \bTheta_h^{\Delta} & = \{ \bxi_h \in L^2(\Omega)^{d \times d} : \left. \bxi_h \right|_T \in RT_1 (T)^d \} \: , \\
    \bW_h^{\Delta} & = \{ \bxi_h \in L^2(\Omega)^d : \left. \bxi_h \right|_T \in RT_1 (T) \} \: .
  \end{split}
\end{equation}
Associated with a triangulation $\cT_h$, we denote by $\cS_h$ the set of all interior sides (edges in 2D and faces in 3D).
For each $\btheta_h^\Delta \in \bTheta_h^\Delta$ (and similarly for $\bw_h^\Delta \in \bW_h^\Delta$) and each interior side $S \in \cS_h$, we define the
jump
\begin{equation}
  \llbracket \btheta_h^\Delta \cdot \bn \rrbracket_S
  = \left. \btheta_h^\Delta \cdot \bn \right|_{T_-} - \left. \btheta_h^\Delta \cdot \bn \right|_{T_+} \: ,
  \label{eq:definition_jump}
\end{equation}
where $\bn$ is the normal direction associated with $S$ (depending on its orientation) and $T_+$ and $T_-$ are the elements
adjacent to $S$ (such that $\bn$ points into $T_+$). We further define $\bZ_h$ or $Z_h$ as the space of
discontinuous $d$-dimensional/scalar vector functions which are piecewise polynomial of degree $1$.
Similarly, $\bQ_h$ stands for the
continuous $d (d-1)/2$-dimensional vector functions which are piecewise polynomial of degree $1$. For every
$d (d-1)/2$-dimensional vector $\btheta$ we define $\bJ^d (\btheta)$ by 
\begin{equation}
  \bJ^2(\theta) := \begin{pmatrix} 0 & \theta \\ -\theta & 0 \end{pmatrix}, 
  \quad \quad \bJ^3(\btheta) :=
  \begin{pmatrix} 0 & \theta_3 & -\theta_2 \\ -\theta_3 & 0 & \theta_1 \\ \theta_2 & -\theta_1 & 0 \end{pmatrix}
\label{eq:skew_symmetric_tensor}
\end{equation}
(cf. \cite[Sect. 9.3]{BofBreFor:13}). Finally, the broken inner product 
\begin{equation}
 ( \: \cdot \: , \: \cdot \: )_h := \sum_{T \in \cT_h} ( \: \cdot \: , \: \cdot \: )_T \: ,
 \label{eq:broken_inner_product}
\end{equation}
will be used, where $( \: \cdot \: , \: \cdot \: )_T$ is the $L^2(T)$ inner product. 

Following the general idea of equilibration (cf. \cite{BraSch:08}) and extending it to the
case of weakly symmetric stresses, the construction is done in the following way:
Compute $\btheta_h^\Delta \in \bTheta_h^\Delta$ in such a way that 
\begin{equation}
  \begin{split}
    ( \div \: \btheta_h^\Delta , \bz_h )_{h}
    & = - ( \bff - \nabla \varphi_h + \div \: \btheta (\bu_h,p_h,\varphi_h) , \bz_h )_{h} \mbox{ for all } \bz_h \in \bZ_h \: , \\
    \langle \llbracket \btheta_h^\Delta \cdot \bn \rrbracket_S , \bzeta \rangle_S
    & = - \langle \llbracket \btheta (\bu_h,p_h) \cdot \bn \rrbracket_S , \bzeta \rangle_S \mbox{ for all }
    \bzeta \in P_1 (S)^d \: , \: S \in \cS_h \: , \\
    ( \btheta_h^\Delta , \bJ^d (\bgamma_h) ) & = 0 \mbox{ for all } \bgamma_h \in \bQ_h \: .
  \end{split}
  \label{eq:equilibration_conditions_stress}
\end{equation}

Due to our specific choice of $\bZ_h$, the first equation in (\ref{eq:equilibration_conditions_stress}) implies that, on
each $T \in \cT_h$, $\div \: \btheta_h^\Delta = - \bff + \nabla \varphi_h - \div \: \btheta_h (\bu_h,p_h,\varphi_h)$ holds.

Similarly, the flux correction $\bw_h^\Delta$ is computed in such a way that
\begin{equation}
  \begin{split}
    ( \tau \: \div \: \bw_h^\Delta , z_h )_{h}
    & = ( g - \tau \: \div \: \bw (\varphi_h) + \frac{p_h - P_h^1 \varphi_h}{\lambda} , z_h )_{h} \mbox{ for all } z_h \in Z_h \: , \\
    \langle \llbracket \bw_h^\Delta \cdot \bn \rrbracket_S , \bzeta \rangle_S
    & = - \langle \llbracket \bw (\varphi_h) \cdot \bn \rrbracket_S , \bzeta \rangle_S \mbox{ for all }
    \bzeta \in P_1 (S) \: , \: S \in \cS_h \: .
  \end{split}
  \label{eq:equilibration_conditions_flux}
\end{equation}

Both $\btheta_h^\Delta$ and $\bw_h^\Delta$ can be constructed locally on vertex patches using the algorithm from Braess and Sch\"oberl
\cite{BraSch:08} for the flux and its weakly symmetric extension from \cite{BerKobMolSta:19} for the stress.

\begin{thm}
  Let $\btheta_h^R \in \bTheta_h^R$ be a stress reconstruction satisfying the weak symmetry condition
  $( \btheta_h^R , \bJ^d (\bgamma_h) ) = 0$ for all $\bgamma_h \in \bQ_h$. Then, for all $\bv \in H_0^1 (\Omega)^d$,
  \begin{equation}
    \left| ( \bas \: \btheta_h^R , \bnabla \bv ) \right|
    \leq C_K \| \bas \: \btheta_h^R \| \: \| \bepsilon ( \bv ) \|
    \label{eq:local_asym}
  \end{equation}
  holds with a constant $C_K$ which depends only on (the largest interior angle in) the triangulation $\cT_h$.
  
  Moreover, for any $s \in L^2 (\Omega)$ with $( s , q_h ) = 0$ for all $q_h \in Q_h$ (the space of conforming piecewise linear functions),
  \begin{equation}
    \left| ( \tr ( \btheta (\bu,p,{\color{black} \varphi})) - \btheta_h^R ) , s ) \right|
    \leq C_D {\color{black} \mu^{1/2}} \| \btheta (\bu,p,{\color{black} \varphi})) - \btheta_h^R \|_{\cA} \: \| s \|
    \label{eq:local_deviator}
  \end{equation}
  holds, where $C_D$ again depends only on (the largest interior angle in) the triangulation $\cT_h$.
  \label{thm-local_bounds}
\end{thm}

\begin{proof}
  The partition of unity
  \begin{equation}
    1 \equiv \sum_{z \in \cV_h} \phi_z \mbox{ on } \Omega
    \label{eq:partition_all}
  \end{equation}
  with respect to the standard hat functions $\phi_z$ associated with the set of all vertices $\cV_h$ in the triangulation $\cT_h$ is used for both inequalities
  (\ref{eq:local_asym}) and (\ref{eq:local_deviator}). This gives rise to the overlapping decomposition of $\Omega$ into the set of vertex patches
  $\omega_z = {\rm supp}\ \phi_z$ for $z \in \cV_h$.
  In order to prove (\ref{eq:local_asym}), the weak symmetry of the reconstructed stress $\btheta_h^R$ implies
  \begin{equation}
    ( \bas \: \btheta_h^R , \bas \: \bnabla \bv ) = ( \bas \: \btheta_h^R , \bas \: \bnabla \: \bv - \bJ^d (\bgamma_h) )
    \mbox{ for all } \bgamma_h = \sum_{z \in \cV_h} \bgamma_z \phi_z
  \end{equation}
  with $\bgamma_z \in \R^{d (d-1)/2}$. From (\ref{eq:partition_all}) we obtain
  \begin{equation}
    \begin{split}
      | ( \bas \: \btheta_h^R , \bnabla \bv ) |
      & = | ( \bas \: \btheta_h^R , \sum_{z \in \cV_h} \left( \bas \: \bnabla \bv - \bJ^d (\bgamma_z) \right) \phi_z ) | \\
      & = \left| \sum_{z \in \cV_h} ( \bas \: \btheta_h^R ,
      \left( \bas \: \bnabla \bv - \bJ^d (\bgamma_z) \right) \phi_z )_{\omega_z} \right| \\
      & = \left| \sum_{z \in \cV_h} ( (\bas \: \btheta_h^R) \phi_z , \bas \: \bnabla \bv - \bJ^d (\bgamma_z) )_{\omega_z} \right| \\
      & \leq \sum_{z \in \cV_h} \| (\bas \: \btheta_h^R) \phi_z \|_{\omega_z} \| \bas \: \bnabla \bv - \bJ^d (\bgamma_z) \|_{\omega_z} \\
      & \leq \sum_{z \in \cV_h} \| \bas \: \btheta_h^R \|_{\omega_z} \| \bas \: \bnabla \bv - \bJ^d (\bgamma_z) \|_{\omega_z} \: .
    \end{split}
    \label{eq:asym_partition}
  \end{equation}
  All rigid body modes $\brho \in \bR\bM$ satisfy $\bnabla \brho = \bJ^d (\bgamma_z)$ with some $\bgamma_z \in \R^{d (d-1)/2}$
  and therefore
  \begin{equation}
    \inf_{\bgamma_z} \| \bas \: \bnabla \bv - \bJ^d (\bgamma_z) \|_{\omega_z}
    \leq \inf_{\brho \in \bR\bM} \| \bas \: \bnabla (\bv - \brho) \|_{\omega_z}
    \leq C_{K,z} \| \bepsilon (\bv) \|_{\omega_z}
    \label{eq:Korn_average}
  \end{equation}
  due to Korn's inequality (cf. \cite{Hor:95}). The constant $C_{K,z}$ obviously only depends on the geometry of the vertex patch
  $\omega_z$ or, more precisely, on its largest interior angle. If we define $C_K = (d+1) \max \{ C_{K,z} : z \in \cV_h \}$, we finally
  obtain from (\ref{eq:asym_partition}) that
  \begin{equation}
    \begin{split}
      | ( \bas \: \btheta_h^R , & \bas \: \bnabla \bv ) |
      \leq \frac{C_K}{d+1} \sum_{z \in \cV_h} \| \bas \: \btheta_h^R \|_{\omega_z} \| \bepsilon (\bv ) \|_{\omega_z} \\
      & \leq C_K \left( \frac{1}{d+1} \sum_{z \in \cV_h} \| \bas \: \btheta_h^R \|_{\omega_z}^2 \right)^{1/2}
      \left( \frac{1}{d+1} \sum_{z \in \cV_h} \| \bepsilon (\bv) \|_{\omega_z}^2 \right)^{1/2} \\
      & = C_K \| \bas \: \btheta_h^R \| \: \| \bepsilon (\bv) \|
    \end{split}
  \end{equation}
  holds, where we used the fact that each element (triangle or tetrahedron) is contained in exactly $d+1$ vertex patches.

  In order to prove (\ref{eq:local_deviator}), the assumption that $s$ is perpendicular to $Q_h$ implies
  \begin{equation}
    \begin{split}
      ( \tr (\btheta (\bu,p,{\color{black} \varphi})) - \btheta_h^R) , s ) & = ( \tr (\btheta (\bu,p,{\color{black} \varphi})) - \btheta_h^R) - \alpha_h , s ) \\
      & \mbox{ for all } \alpha_h = \sum_{z \in \cV_h} \alpha_z \phi_z \: , \: \alpha_z \in \R \: .
    \end{split}
  \end{equation}
  Using the partition of unity (\ref{eq:partition_all}) once more, we obtain
  \begin{equation}
    \begin{split}
      ( \tr (\btheta (\bu,p,{\color{black} \varphi})) - \btheta_h^R) , s )
      & = \sum_{z \in \cV_h} ( (\tr (\btheta (\bu,p,{\color{black} \varphi})) - \btheta_h^R) - \alpha_z) \phi_z , s )_{\omega_z} \\
      & = \sum_{z \in \cV_h} ( \tr (\btheta (\bu,p,{\color{black} \varphi})) - \btheta_h^R) - \alpha_z , s )_{\omega_z} .
    \end{split}
    \label{eq:volum_partition}
  \end{equation}
  
  We choose $\alpha_z$ in such a way that $( \tr (\btheta (\bu,p,{\color{black} \varphi})) - \btheta_h^R) - \alpha_z , 1 )_{\omega_z} = 0$
  and use the ``dev-div lemma'' (cf. \cite[Prop. 9.1.1]{BofBreFor:13}) to obtain
  \begin{equation}
    \| \tr \: (\btheta (\bu,p,{\color{black} \varphi})) - \btheta_h^R) - \alpha_z \|_{\omega_z}
    \leq C_{D,z} \| \bdev  (\btheta (\bu,p,{\color{black} \varphi})) - \btheta_h^R) \|_{\omega_z} \: ,
    \label{eq:dev_div_average}
  \end{equation}
  where $\bdev \: \bxi = \bxi - (1/d) (\tr \: \bxi) \: \bI$ denotes the deviatoric part of \bxi
  (note that $\div (\btheta (\bu,p,{\color{black} \varphi})) - \btheta_h^R) = \bzero$ is satisfied). Obviously, $C_{D,z}$ depends only on the shape of $\omega_z$.
  This leads to
  \begin{equation}
    \begin{split}
      \left| ( \tr (\btheta (\bu,p,{\color{black} \varphi})) - \btheta_h^R) - \alpha_z , s \phi_z )_{\omega_z} \right|
      & \leq \| \tr (\btheta (\bu,p,{\color{black} \varphi})) - \btheta_h^R) - \alpha_z \|_{\omega_z} \| s \phi_z \|_{\omega_z} \\
      & \leq C_{D,z} \| \bdev (\btheta (\bu,p,{\color{black} \varphi})) - \btheta_h^R) \|_{\omega_z} \| s \|_{\omega_z} \: .
    \end{split}
  \end{equation}
  Setting $C_D = \sqrt{2} (d+1) \max \{ C_{D,z} : z \in \cV_h \}$, we obtain from (\ref{eq:volum_partition}) that
  \begin{equation}
    \begin{split}
      & | ( \tr (\btheta (\bu,p,{\color{black} \varphi})) - \btheta_h^R) , s ) |
      \leq \sum_{z \in \cV_h} C_{D,z} \| \bdev (\btheta (\bu,p,{\color{black} \varphi})) - \btheta_h^R) \|_{\omega_z} \| s \|_{\omega_z} \\
      & \leq \frac{C_D}{d+1} \left( \sum_{z \in \cV_h} \frac{1}{2 \mu} \| \bdev (\btheta (\bu,p,{\color{black} \varphi})) - \btheta_h^R) \|_{\omega_z}^2 \right)^{1/2}
      \left( \sum_{z \in \cV_h} \| s \|_{\omega_z}^2 \right)^{1/2} \\
      & = C_D \frac{1}{\sqrt{2}} \| \bdev (\btheta (\bu,p,{\color{black} \varphi})) - \btheta_h^R) \| \: \| s \|
      \leq C_D {\color{black} \mu^{1/2}} \| \btheta (\bu,p,{\color{black} \varphi})) - \btheta_h^R \|_\cA \: \| s \|
    \end{split}
  \end{equation}
  holds, where the inequality
  \begin{equation}
    \| \bdev (\btheta (\bu,p,{\color{black} \varphi})) - \btheta_h^R) \| \leq \sqrt{2 \mu} \| \btheta (\bu,p,{\color{black} \varphi})) - \btheta_h^R \|_\cA
  \end{equation}
  was used.
\end{proof}

\begin{figure}[t]
    \centering
    \includegraphics[width=\textwidth]{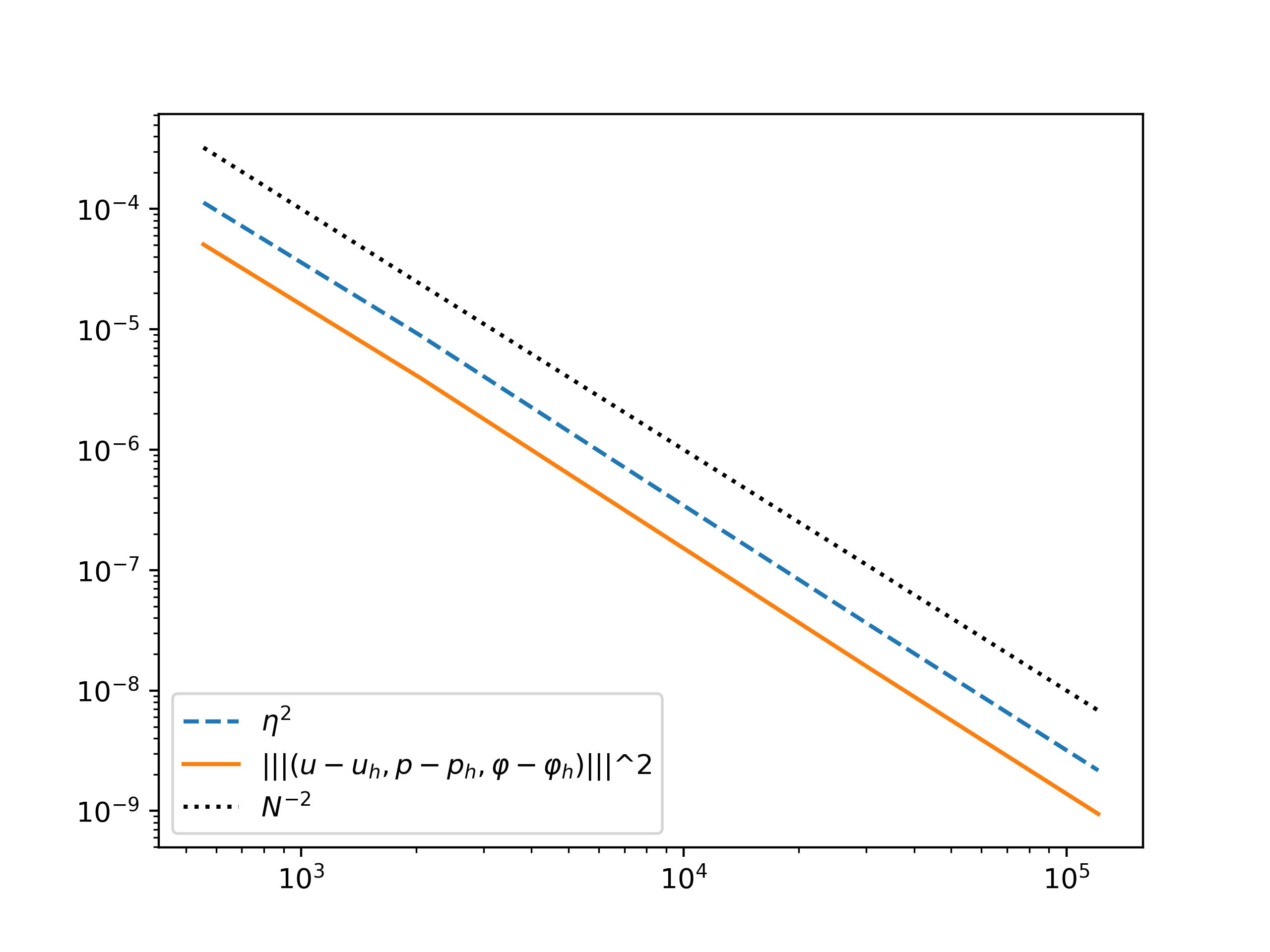}
    \caption{Comparison of  error estimator and error (regular solution)}
    \label{fig:sums-uniform}
\end{figure}

\begin{remark}
  Theorem \ref{thm-local_bounds} finally justifies the use of (\ref{eq:local_bounds_both}) in the proof of Lemma \ref{lem-stress_bound}
  since $\bu - \bu_h \in H_0^1 (\Omega)^d$ and since
  \begin{equation}
    ( \frac{p_h - \varphi_h}{\lambda} + {\rm div} \: \bu_h , q_h ) = 0 \mbox{ for all } q_h \in Q_h
    \label{eq:perpendicular}
  \end{equation}
  follows from the second equation in (\ref{eq:Biot_system_discrete}).
\end{remark}

\begin{remark}
For the proof of local efficiency of the error estimator, one would proceed in the usual way by bounding
\begin{align}
\label{eq:est}
    \eta = \left( \eta_S^2+\eta_A^2+\eta_C^2+\eta_F^2+\eta_P^2 \right)^{1/2}
\end{align}
from above with suitable residual estimator quantities. This can be done following the lines of the proof in \cite{BerKobMolSta:19} using the vertex
patch  decompositions
\(\left\|\boldsymbol{\theta}_{h}^{\Delta}\right\| \leq \sum_{z \in \mathcal{V}_{h}}
\left\|\boldsymbol{\theta}_{h, z}^{\Delta}\right\| \omega_{z}\)
and
\(\left\|\boldsymbol{w}_{h}^{\Delta}\right\| \leq 
\sum_{z \in \mathcal{V}_{h}}
\left\|\boldsymbol{w}_{h, z}^{\Delta}\right\| \omega_{z} \ .\)
While the local efficiency of the arising residual estimators is known for the elasticity and flow subproblems, the term $\eta_C$
measuring the constraint seems to require special care.
\end{remark}

\begin{figure}[t]
    \centering
    \includegraphics[width=\textwidth]{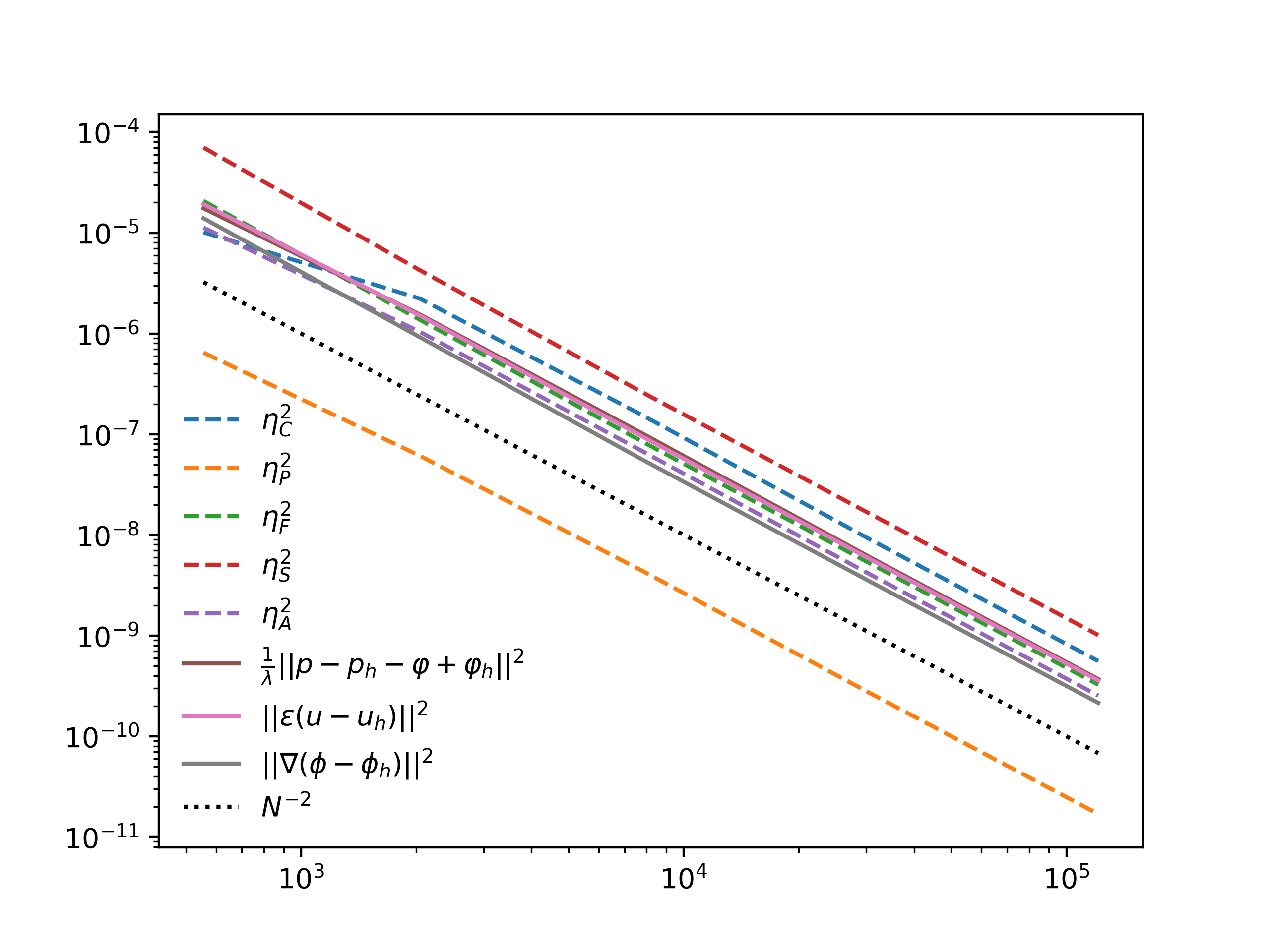}
    \caption{Comparison of error estimator and error (regular solution)}
    \label{fig:uniform}
\end{figure}

\section{Numerical Results}

In this section, the theoretical results are illustrated on several test cases. The first one is designed to possess a regular analytical solution on the unit
square $\Omega = [0,1]^2$ combined with zero boundary conditions on $\partial \Omega$ for both the displacements $\Disp$ and the fluid-pressure
$\varphi$. Choosing $\varphi(x,y) = xy(1-x)(1-y)$ and 
$\Disp(x,y) = (\varphi(x,y),\varphi(x,y))^\top$ leads to
$$p(x,y) = -\lambda
\left(\frac{\partial}{\partial x}
\varphi(x,y)
+
\frac{\partial}{\partial y}
\varphi(x,y)
\right)+\varphi(x,y) \ .$$
The source terms corresponding to these exact solutions, see \ref{fig:uniform-solutions},  are now given by
\begin{align*}
    {\bff} = \nabla \varphi-
    \begin{pmatrix}
     (3\mu+\lambda) \frac{\partial^2}{\partial x^2} \varphi &
     (\mu+\lambda) \frac{\partial^2}{\partial x\partial y} \varphi
     \\
     (\mu+\lambda) \frac{\partial^2}{\partial x\partial y} \varphi
     &(3\mu+\lambda) \frac{\partial^2}{\partial y^2} \varphi
    \end{pmatrix}
\end{align*}
and
\begin{align*}
    {g} = \frac{\partial}{\partial x}
\varphi(x,y)
+
\frac{\partial}{\partial y}
\varphi(x,y)
-\tau\div\nabla\varphi
\end{align*}
We compare the estimated error with the analytical error in figure \ref{fig:sums-uniform}. The estimated error is given by \eqref{eq:est} while the analytical
error is measured for the fluid-pressure, the total-pressure and the displacements in the norms defined by \eqref{eq:energy_norm}.
The graphs show a doubly logarithmic plot in terms of the degrees of freedom $N$ and,
as expected, both curves show the quadratic convergence order. We also compare the convergence rates of each part of the error estimator in figure
\ref{fig:uniform} with each part of the analytical error, and observe the expected quadratic convergence as well.

\begin{figure}[t]
\centering
\begin{tabular}{l l }
    \begin{subfigure}{.6\textwidth}\centering
    \includegraphics[width=\textwidth,trim=24mm 0mm 24mm 5mm]{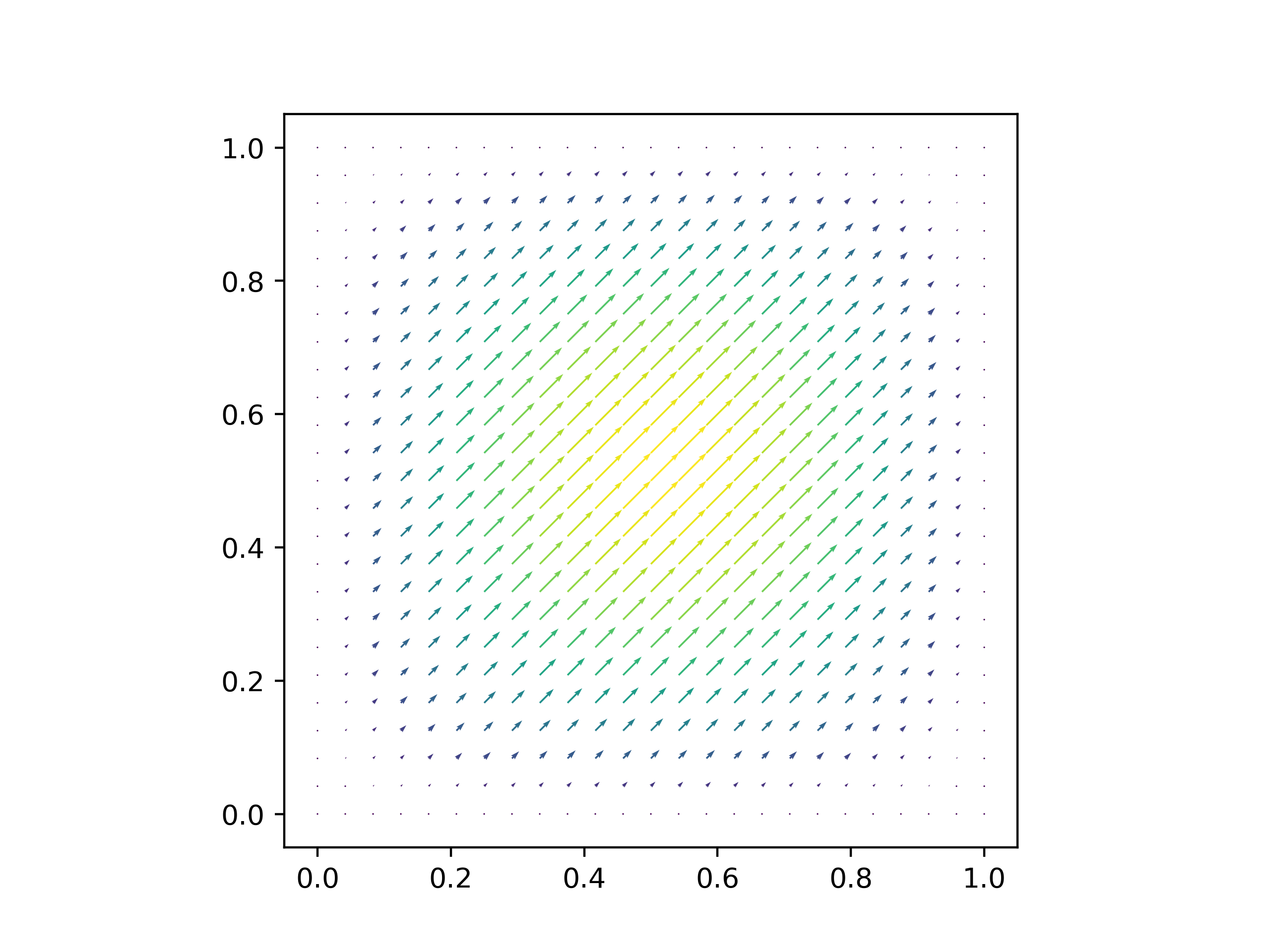}
    \caption{u}
    \end{subfigure}
    $ $\\[-76mm]
    &
    \begin{subfigure}{.3\textwidth}
    \includegraphics[width=\textwidth,trim=24mm 12mm 24mm 5mm, clip]{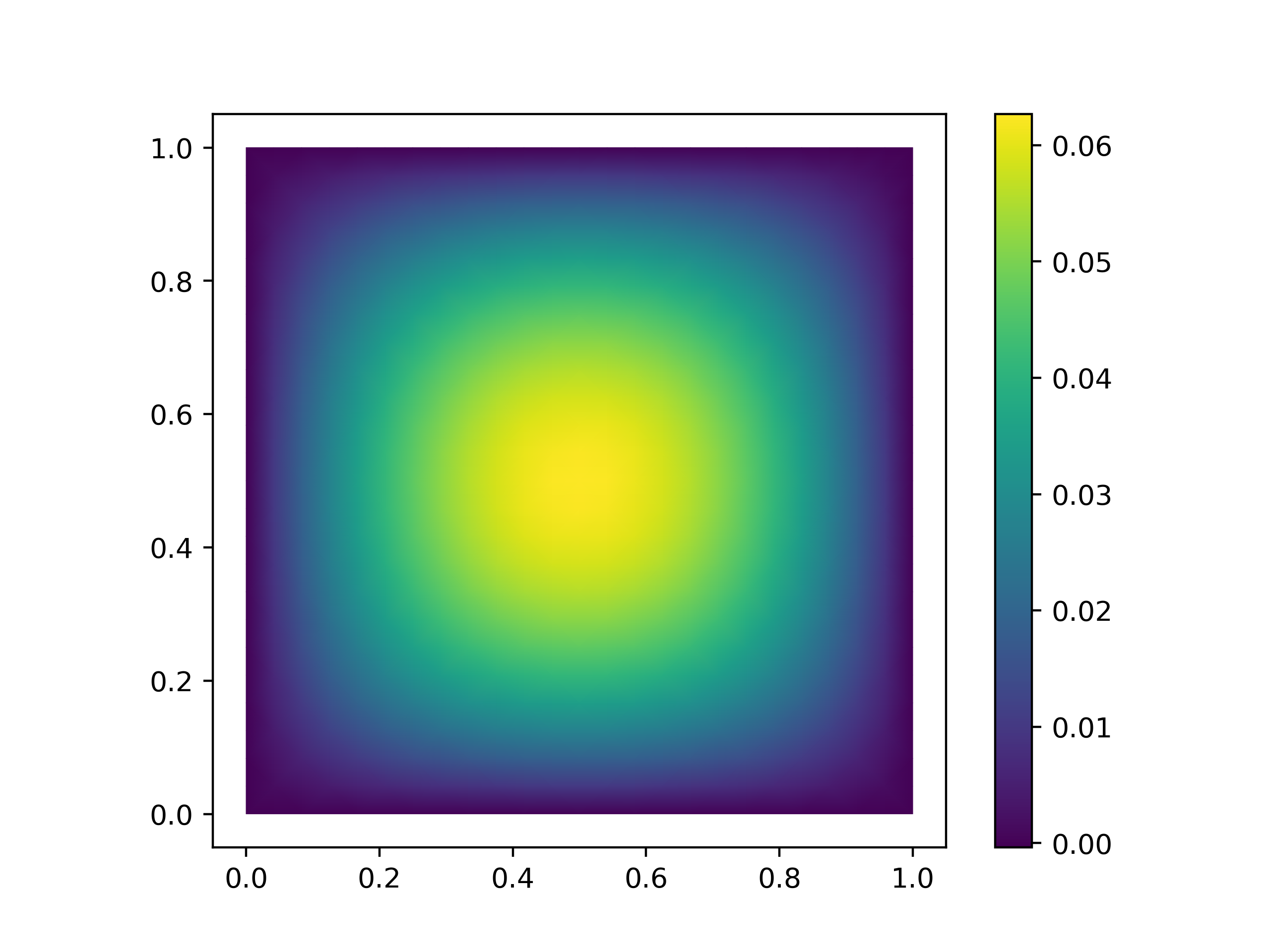}
    \caption{phi}
    \end{subfigure}
\\
    &
    \begin{subfigure}{.3\textwidth}
    \includegraphics[width=\textwidth,trim=24mm 12mm 24mm 5mm, clip]{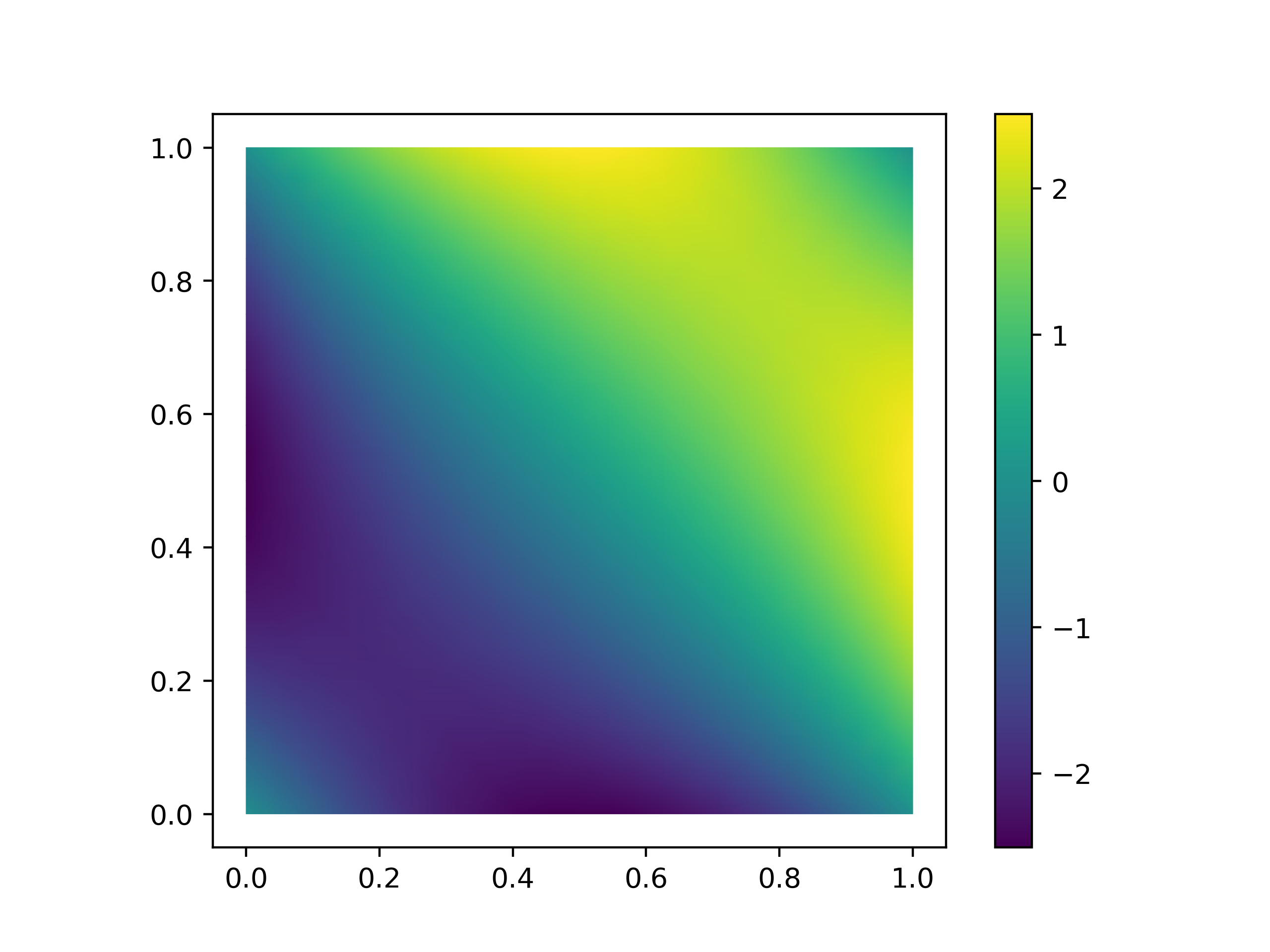}
    \caption{p }
    \end{subfigure}
\end{tabular}
    \caption{Approximations of the regular solution on third uniform refinement}
    \label{fig:uniform-solutions}
\end{figure}

In order to investigate the robustness with respect to the material parameter, we let $\lambda$ vary over several order of magnitude. Note that for this
analytical test, the norm of the exact solution scales with $\lambda$ when $\lambda$ tends to infinity. In fact, we have $\| \varphi \| = 30^{-1}$ and
$\| \nabla  \varphi \|^2 =  20 \|  \varphi \|^2$ as well as 
%\begin{align*}
%\| \nabla  \varphi \| &= \|  \frac{\partial}{\partial x}
%\varphi(x,y)
%+
%\frac{\partial}{\partial y}
%\varphi(x,y)\| 
%\\
%\text{and }
\[
  \| \epsilon(u)  \|^2 = \frac 3 2  \| \nabla  \varphi \|^2  \ ,
\]
 %\end{align*}
so that 
 \begin{align*}
 ||| (\Disp,\varphi,p)|||^2 = \left( \tau + \frac 3 2 \mu + 20\lambda +2\lambda^{-1}\right) \| \varphi \|^2
\ .  \end{align*}
Nevertheless, we can observe the robustness of the a posteriori error estimates in figure \ref{fig:sums-uniform-tau1}.
\begin{figure}[t]
    \centering
    \includegraphics[width=\textwidth]{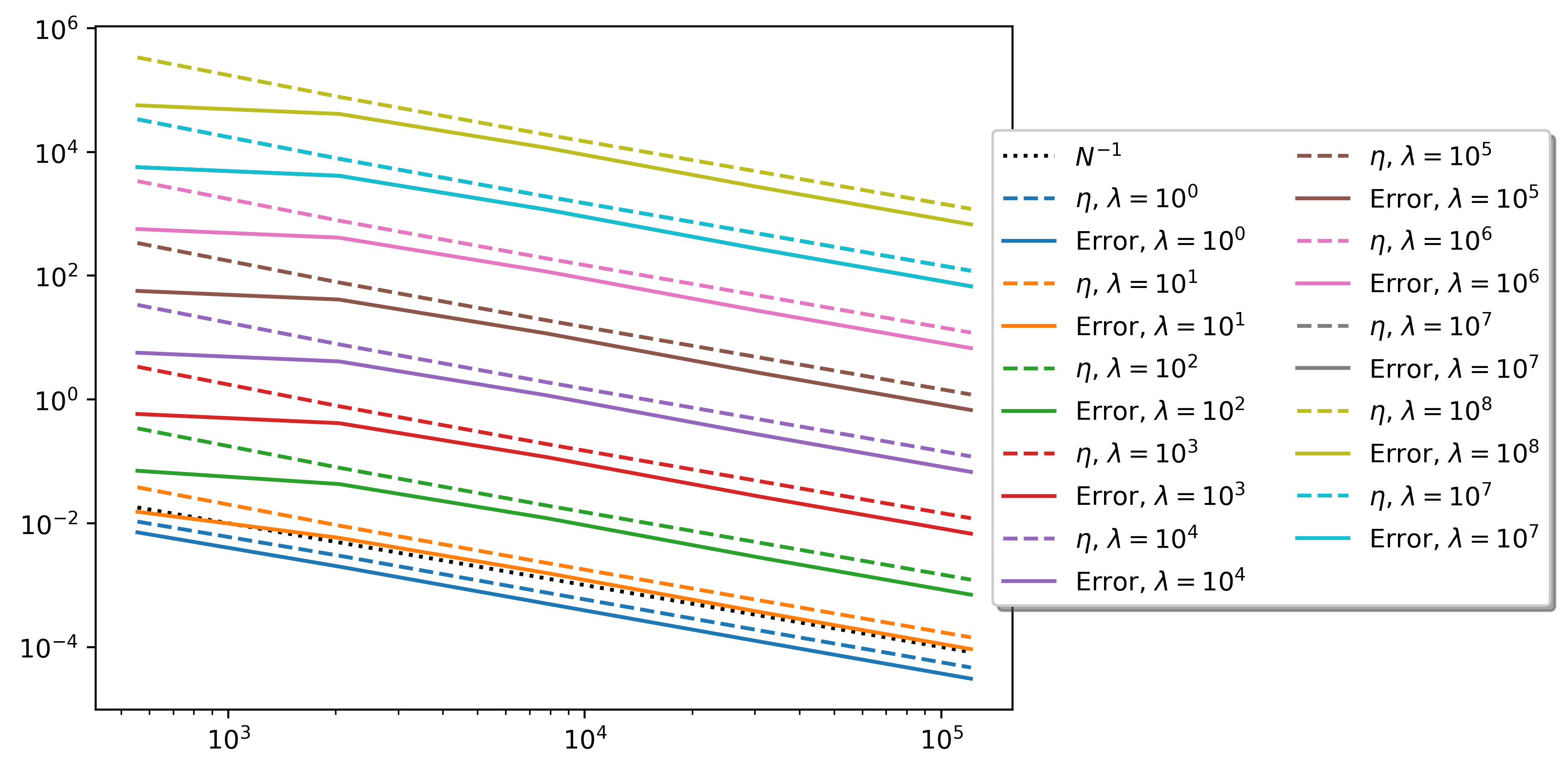}
    \caption{Comparison of error estimator and error for several values of $\lambda$ and $\tau=\mu=1$}
    \label{fig:sums-uniform-tau1}
\end{figure}

The second test case should illustrate the use of the error indicator in an adaptive strategy. Therefore, we consider the L-shaped domain
$\Omega = [-1,1]^2\backslash [0,1]^2$, constant source terms $\bff = (1,1)^\top$ and $g=1$ and impose zero Dirichlet boundary conditions on the fluid
pressure and on the displacements. It is well-known that the fluid pressure and the displacements solution corresponding to these parameter are not
$H^2(\Omega)$-regular. The distribution of the error on the initial mesh for $\mu=\lambda=\tau$ is shown in figure \ref{fig:initialmesh}.
\begin{figure}[t]
    \centering
    \includegraphics[width=0.45\textwidth,trim=20mm 12mm 2mm 5mm]{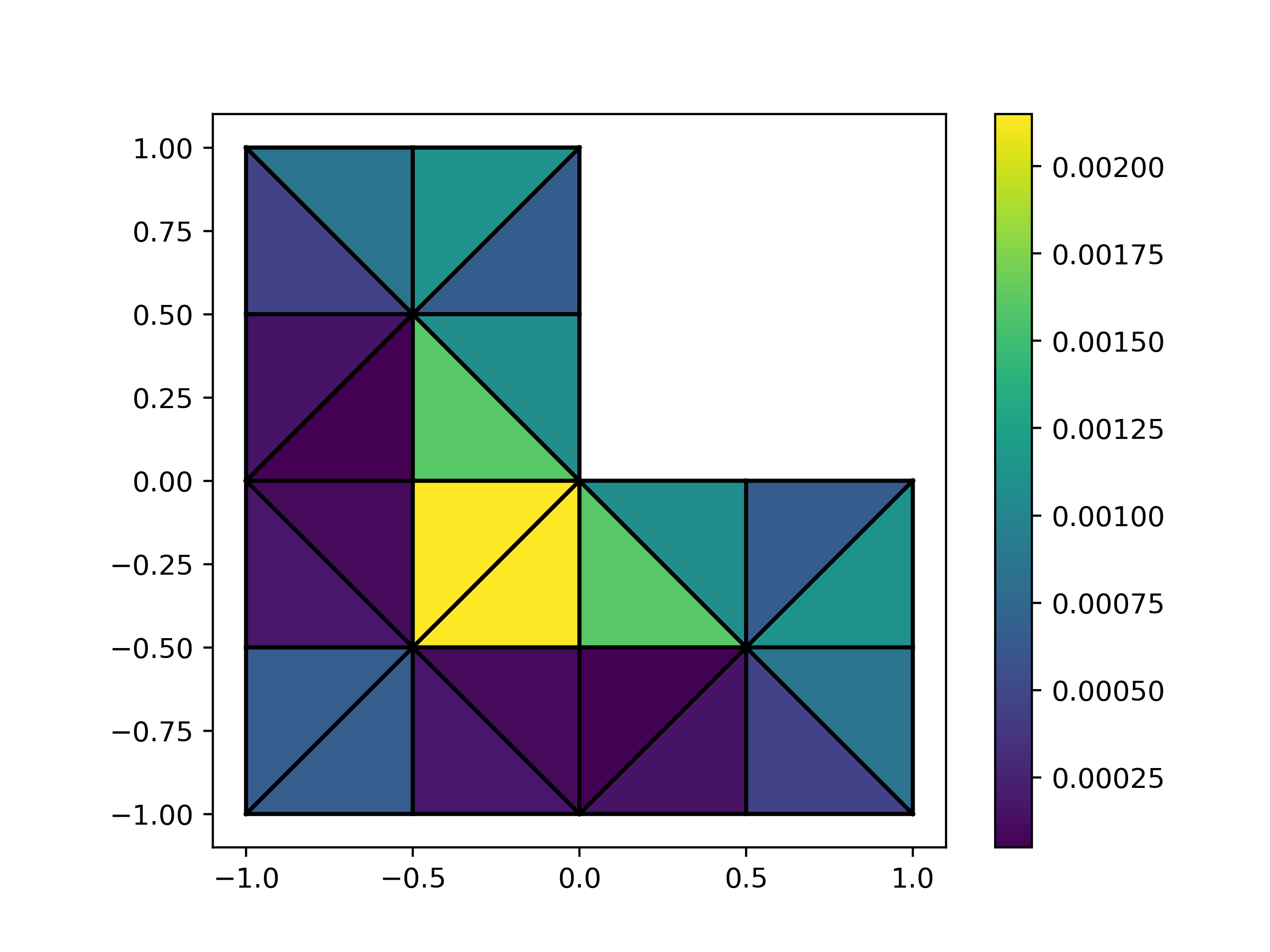}
    \includegraphics[width=0.45\textwidth,trim=20mm 12mm 2mm 5mm]{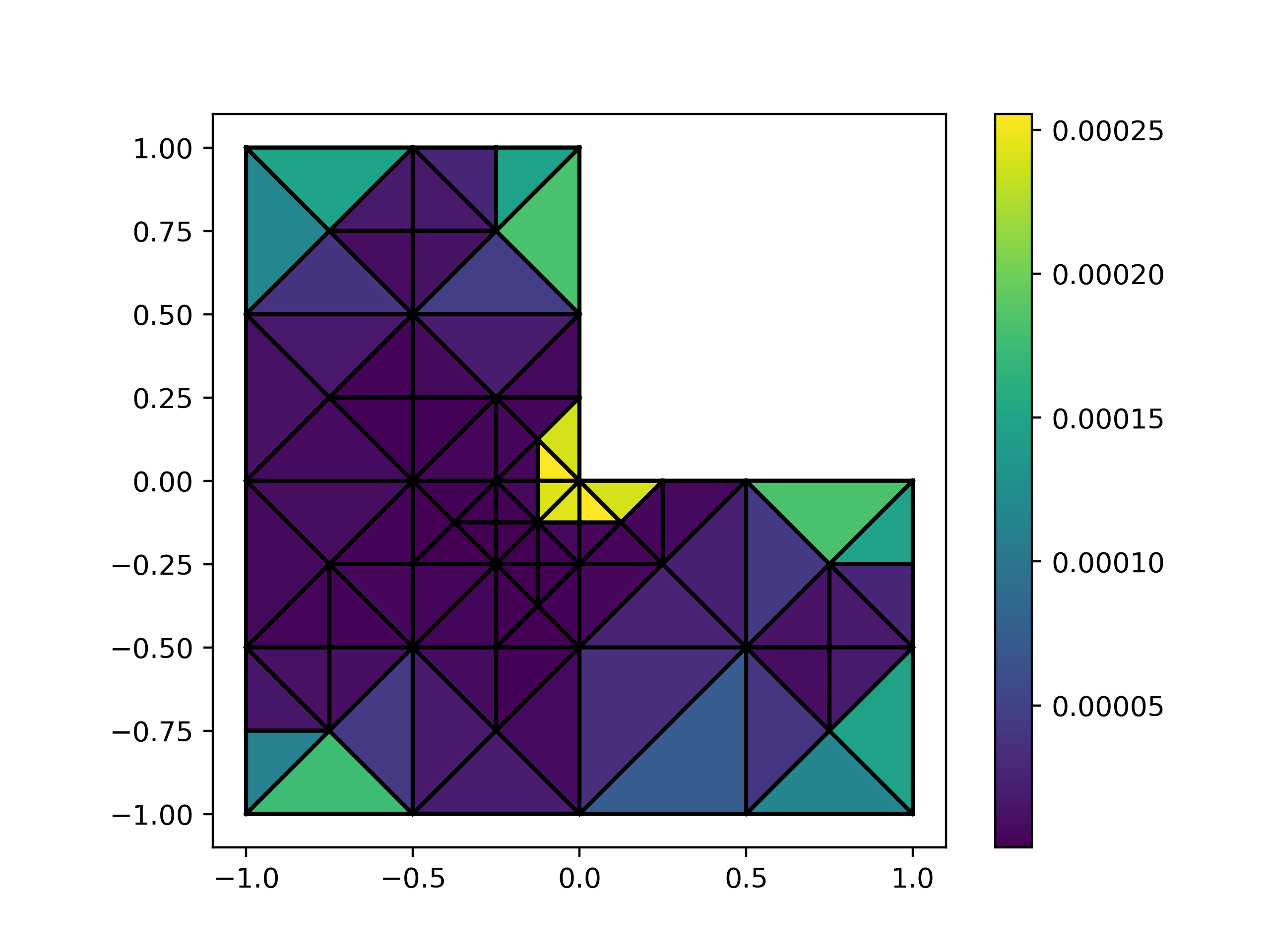}
    \caption{Repartition of the error on the initial and the fifth mesh for $\lambda=\tau=\mu=1$}
    \label{fig:initialmesh}
\end{figure}
We used the Dörfler marking strategy with several Dörfler parameters and an overkill solution computed on a finer mesh with cubic polynomials. As expected,
we observe that the adaptive algorithm is able to recover the optimal
convergence rate in figure \ref{fig:adaptiv-doerfler}, while the uniform refinement leads to a slower convergence rate due to the reentrant corner.
Although the optimal convergence rate is usually achieved independently of the D\"orfler parameter, the computational results on a finite
number of refinement levels may actually be sensitive to it, depending on the kind of singularities involved in the solution. We let
$\lambda$ vary from $1$ to $10^8$ and $\tau$ from $1$ to $10^{-8}$ and observed the robustness of the a posteriori estimate.
%The solutions $\varphi$ and $p$ are shown in figure \ref{fig:adasolu}.
\begin{figure}[h]
    \centering
    \includegraphics[width=\textwidth]{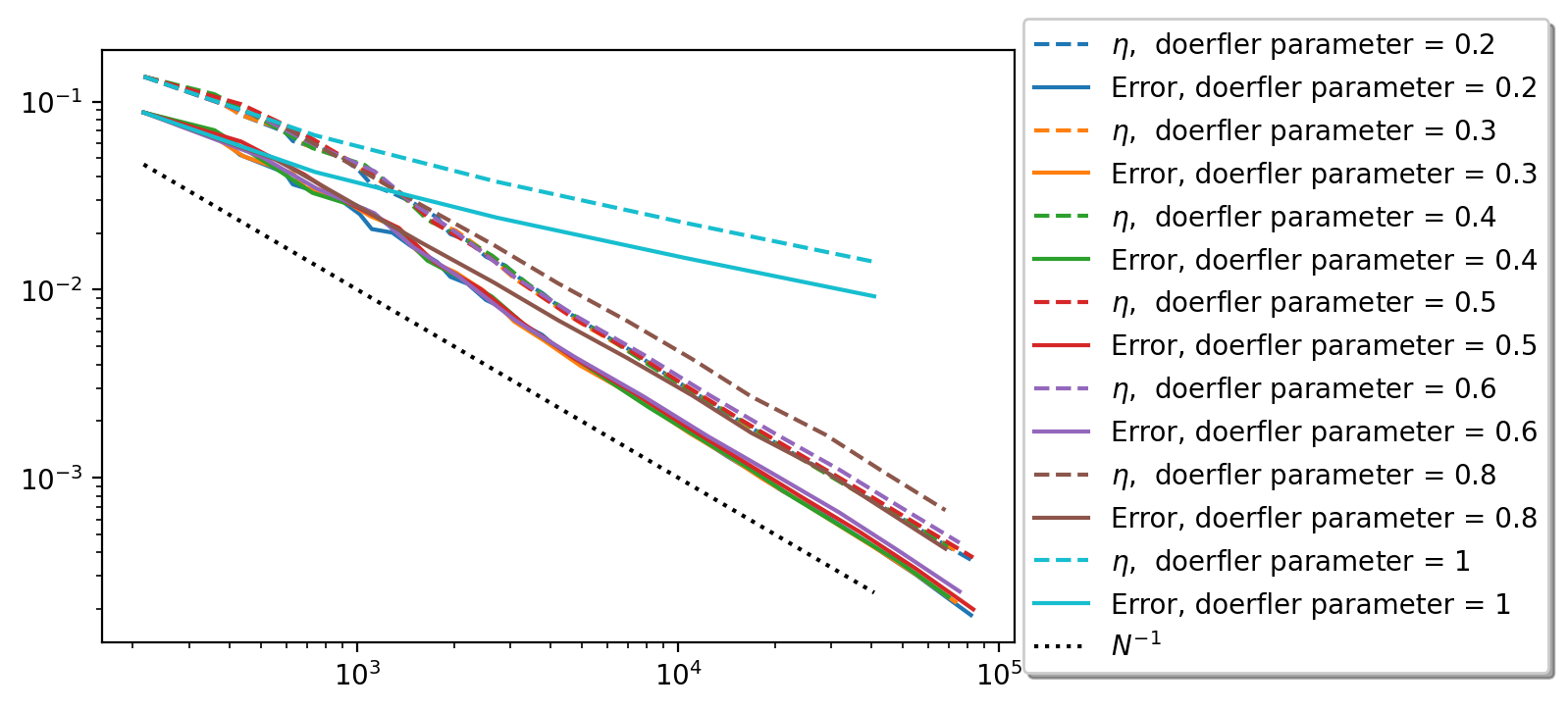}
    \caption{Comparison of error estimator and error for uniform and adaptive refinement}
    \label{fig:adaptiv-doerfler}
\end{figure}

%\begin{figure}[h]
%    \centering
%    \includegraphics[width=0.45\textwidth,trim=20mm 10mm 0mm 05mm]{./phi1ada}
%    \includegraphics[width=0.45\textwidth,trim=20mm 10mm 0mm 05mm]{./pada}
%    \caption{Solutions $\varphi$ and $p$ for $\lambda=10^8$ and $\tau = 10^{-8}$}
%    \label{fig:adasolu}
%\end{figure}

{\em Acknowledgement.} We thank the two anonymous referees for the very careful reading of our manuscript and valuable suggestions that led to a considerable
improvement of the presentation.

\bibliography{articles,books}

\begin{thebibliography}{10}
\expandafter\ifx\csname url\endcsname\relax
  \def\url#1{\texttt{#1}}\fi
\expandafter\ifx\csname urlprefix\endcsname\relax\def\urlprefix{URL }\fi
\expandafter\ifx\csname href\endcsname\relax
  \def\href#1#2{#2} \def\path#1{#1}\fi

\bibitem{Bio:41}
M.~A. Biot, General theory of three-dimensional consolidation, J. Appl. Phys.
  12 (1941) 155--164.

\bibitem{MurThoLou:96}
M.~Murad, V.~Thom{\'e}e, A.~Loula, Asymptotic behavior of semidiscrete
  finite-element approximations of {B}iot's consolidation problem, SIAM J.
  Numer. Anal. 33 (1996) 1065--1083.

\bibitem{PhiWhe:08}
P.~J. Phillips, M.~F. Wheeler, A coupling of mixed and discontinuous {G}alerkin
  finite element methods for poroelasticity, Comput. Geosci. 12 (2008)
  417--435.

\bibitem{OyaRui:16}
R.~Oyarz{\'u}a, R.~Ruiz-Baier, Locking-free finite element methods for
  poroelasticity, SIAM J. Numer. Anal. 54 (2016) 2951--2973.

\bibitem{LeeMarWin:17}
J.~J. Lee, K.-A. Mardal, R.~Winther, Parameter-robust discretization and
  preconditioning of {B}iot's consolidation model, SIAM J. Sci. Comput. 39
  (2017) A1--A24.

\bibitem{ErnMeu:09}
A.~Ern, S.~Meunier, A posteriori error analysis of {E}uler-{G}alerkin
  approximations to coupled elliptic-parabolic problems, ESAIM: M2AN 43 (2009)
  353--375.

\bibitem{KorSta:05}
J.~Korsawe, G.~Starke, A least squares mixed finite element method for {B}iot's
  consolidation problem in porous media, SIAM J. Numer. Anal. 43 (2005)
  318--339.

\bibitem{Lee:16}
J.~J. Lee, Robust error analysis of coupled mixed methods for {B}iot's
  consolidation model, J. Sci. Comput. 69 (2016) 610--632.

\bibitem{RieDipErn:17}
R.~Riedlbeck, D.~A. {DiPietro}, A.~Ern, Equilibrated stress reconstructions for
  linear elasticity problems with application to a posteriori error analysis,
  in: C.~Canc{\`e}s, P.~Omnes (Eds.), Finite Volumes for Complex Applications
  VIII -- Methods and Theoretical Aspects, Springer, 2017, pp. 293--301.

\bibitem{PraSyn:47}
W.~Prager, J.~L. Synge, Approximations in elasticity based on the concept of
  function space, Quart. Appl. Math. 5 (1947) 241--269.

\bibitem{LadLeg:83}
P.~Ladev\`{e}ze, D.~Leguillon, Error estimate procedure in the finite element
  method and applications, SIAM J. Numer. Anal. 20 (1983) 485--509.

\bibitem{BraSch:08}
D.~Braess, J.~Sch{\"o}berl, Equilibrated residual error estimator for edge
  elements, Math. Comp. 77 (2008) 651--672.

\bibitem{BraPilSch:09}
D.~Braess, V.~Pillwein, J.~Sch{\"o}berl, Equilibrated residual error estimates
  are $p$-robust, Comput. Methods Appl. Mech. Engrg. 198 (2009) 1189--1197.

\bibitem{ErnVoh:15}
A.~Ern, M.~Vohral{\'{\i}}k, Polynomial-degree-robust a posteriori error
  estimates in a unified setting for conforming, nonconforming, discontinuous
  {G}alerkin, and mixed discretizations, SIAM J. Numer. Anal. 53 (2015)
  1058--1081.

\bibitem{ArnWin:02}
D.~N. Arnold, R.~Winther, Mixed finite elements for elasticity, Numer. Math. 92
  (2002) 401--419.

\bibitem{BerKobMolSta:19}
F.~Bertrand, B.~Kober, M.~Moldenhauer, G.~Starke, Weakly symmetric stress
  equilibration and a posteriori error estimation for linear elasticity,
  Submitted for PublicationArXiv: 1808.02655.

\bibitem{BofBreFor:13}
D.~Boffi, F.~Brezzi, M.~Fortin, Mixed Finite Element Methods and Applications,
  Springer, Heidelberg, 2013.

\bibitem{AinVej:19}
M.~Ainsworth, T.~Vejchodsky, A simple approach to reliable and robust a
  posteriori error estimation for singularly perturbed problems, Comput.
  Methods Appl. Mech. Engrg. 353 (2019) 373--390.

\bibitem{Hor:95}
C.~O. Horgan, Korn's inequalities and their applications in continuum
  mechanics, SIAM Rev. 37 (1995) 491--511.

\end{thebibliography}
 
\begin{figure}[h]
    \centering
    \includegraphics[width=\textwidth]{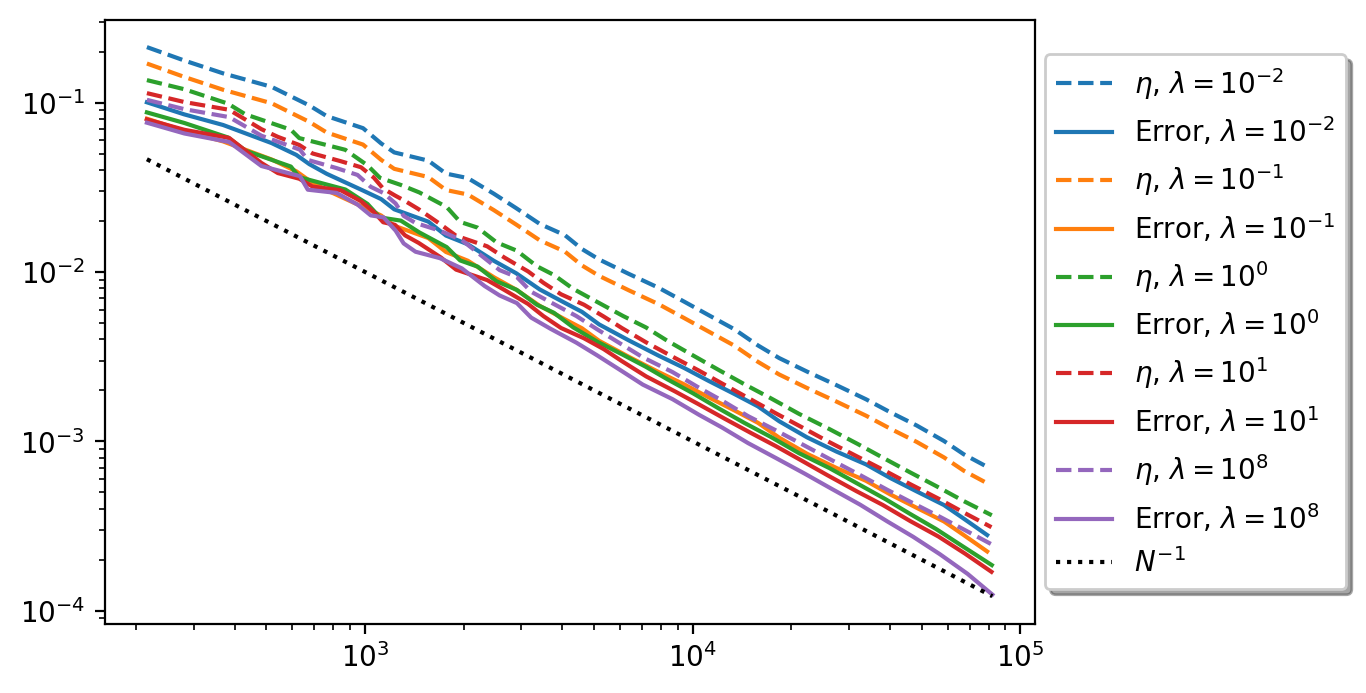}
    \caption{Error estimator and error for adaptive refinement, dörfler parameter 0.2}
    \label{fig:adaptiv-doerfler20}
\end{figure}
\begin{figure}[h]
    \centering
    \includegraphics[width=\textwidth]{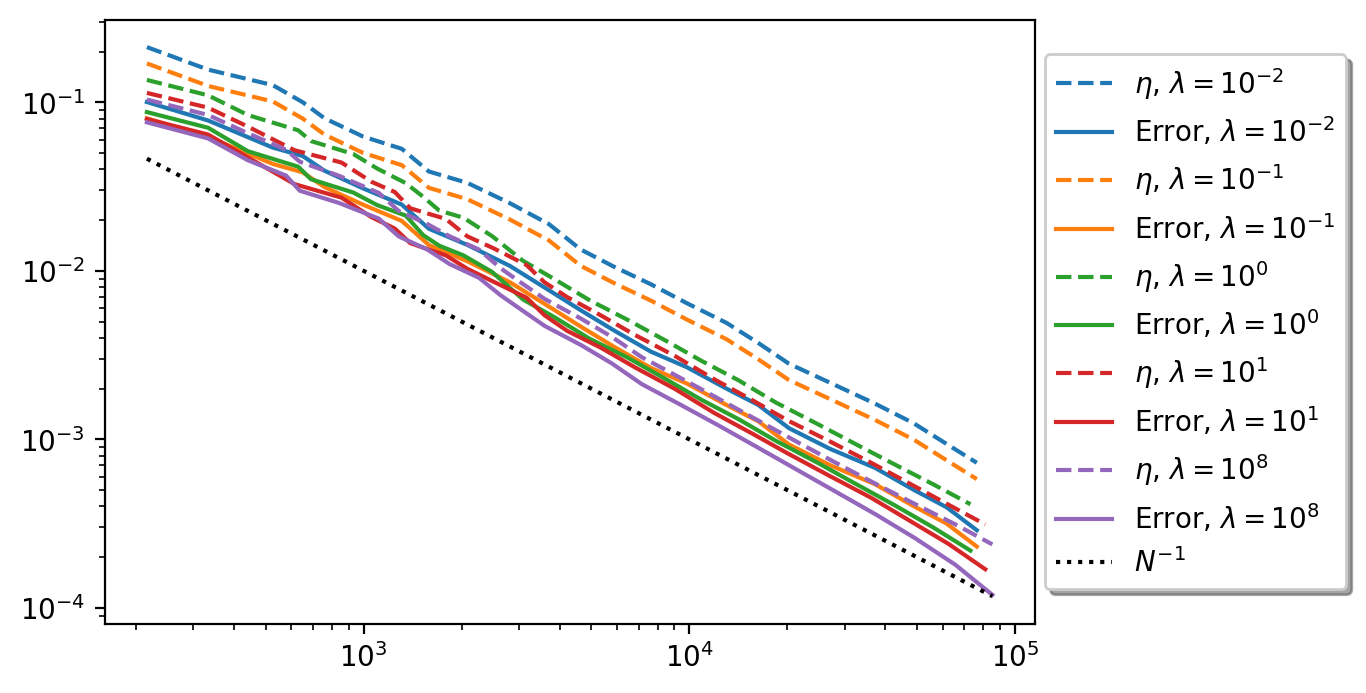}
    \caption{Error estimator and error for adaptive refinement, dörfler parameter 0.3}
    \label{fig:adaptiv-doerfler30}
\end{figure}
\begin{figure}[h]
    \centering
    \includegraphics[width=\textwidth]{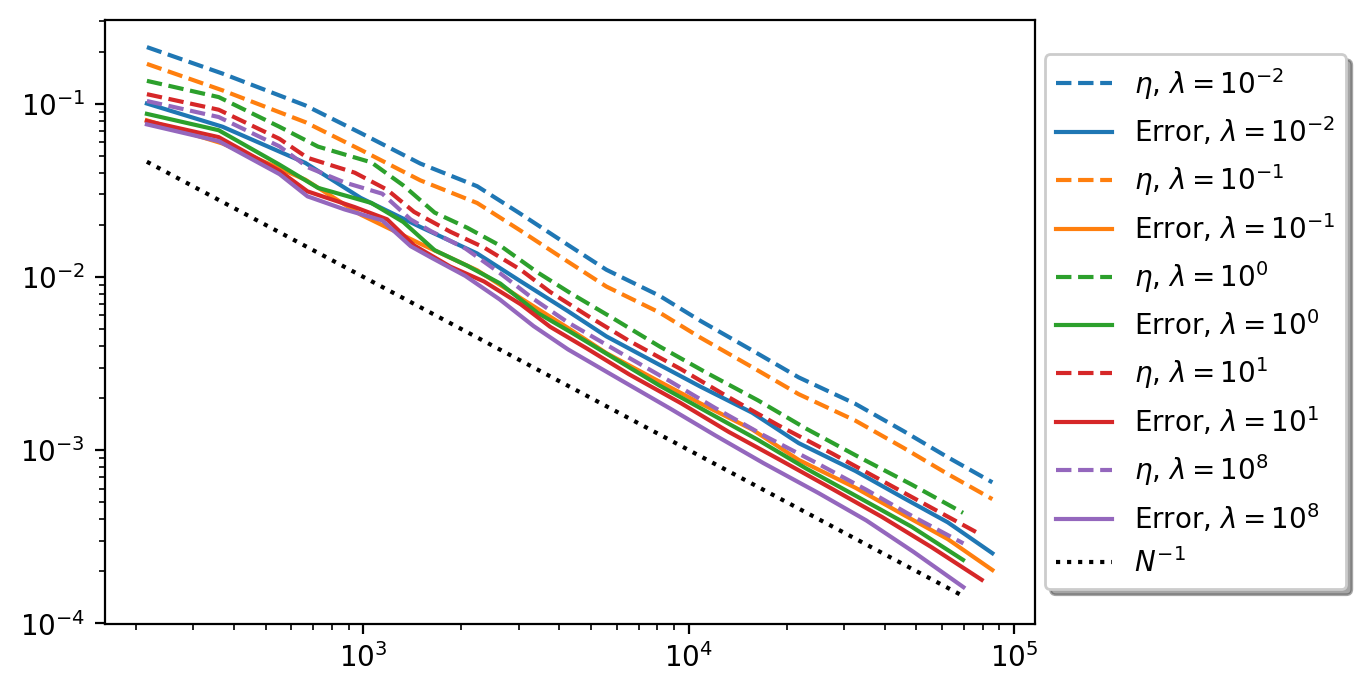}
    \caption{Error estimator and error for adaptive refinement, dörfler parameter 0.4}
    \label{fig:adaptiv-doerfler40}
\end{figure}
\begin{figure}[h]
    \centering
    \includegraphics[width=\textwidth]{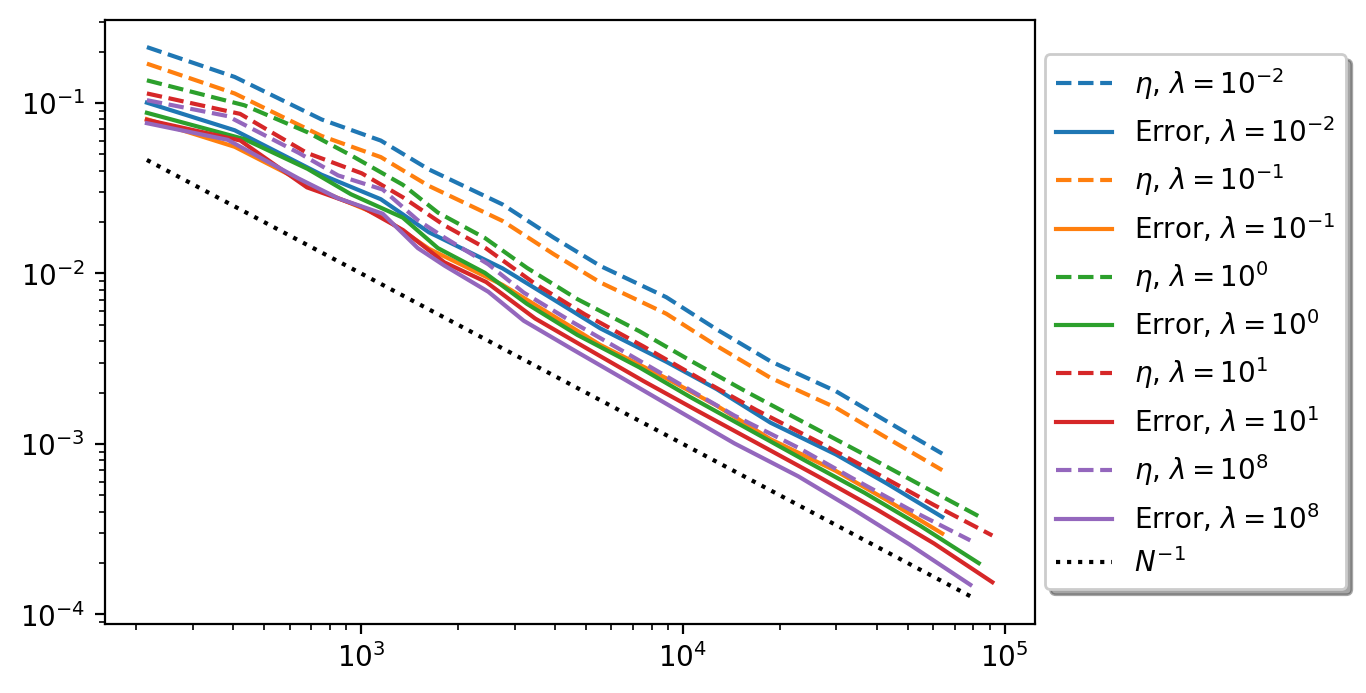}
    \caption{Error estimator and error for adaptive refinement, dörfler parameter 0.5}
    \label{fig:adaptiv-doerfler50}
\end{figure}
\begin{figure}[h]
    \centering
    \includegraphics[width=\textwidth]{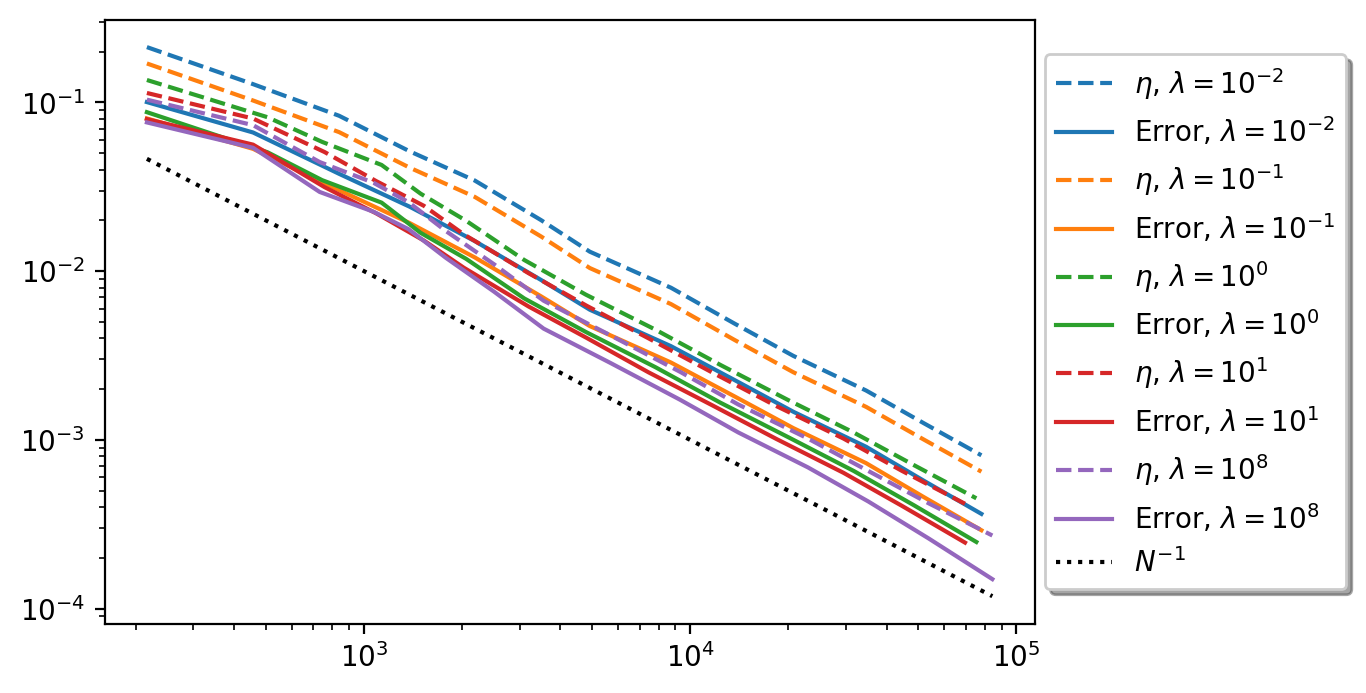}
    \caption{Error estimator and error for adaptive refinement, dörfler parameter 0.6}
    \label{fig:adaptiv-doerfler60}
\end{figure}
\begin{figure}[h]
    \centering
    \includegraphics[width=\textwidth]{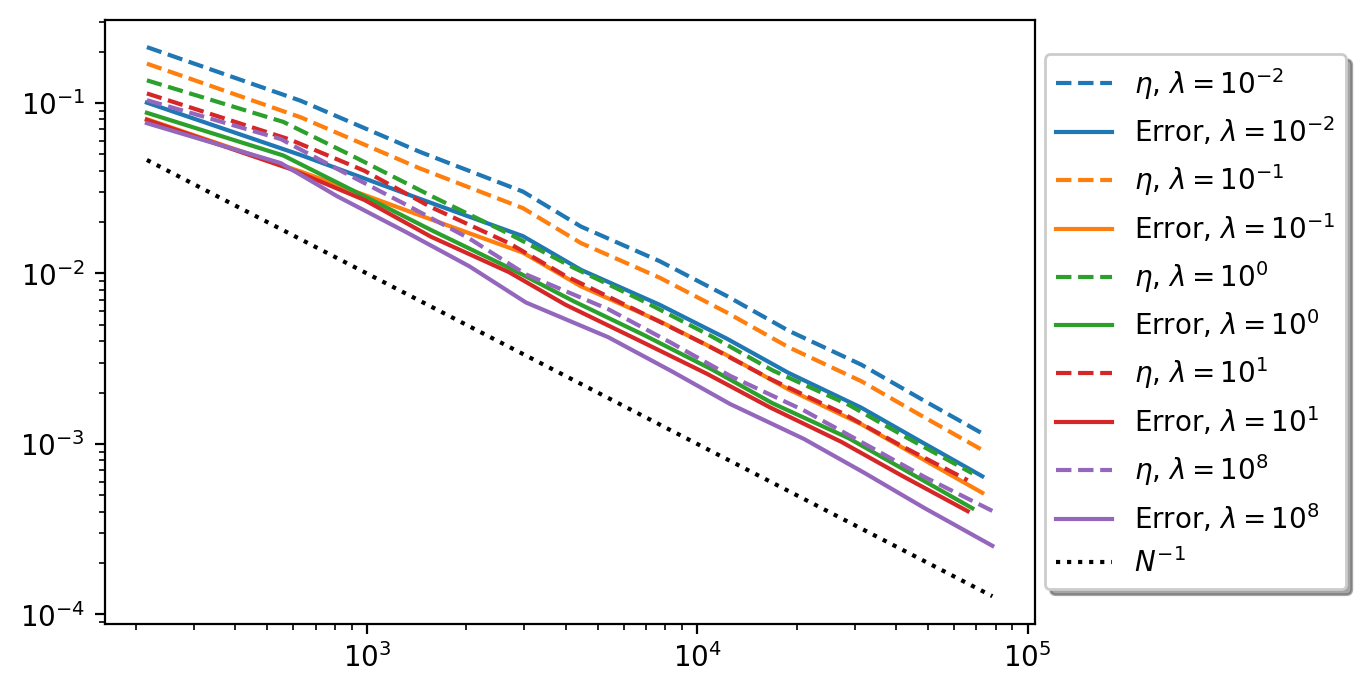}
    \caption{Error estimator and error for adaptive refinement, dörfler parameter 0.8}
    \label{fig:adaptiv-doerfler80}
\end{figure}
\begin{figure}[h]
    \centering
    \includegraphics[width=\textwidth]{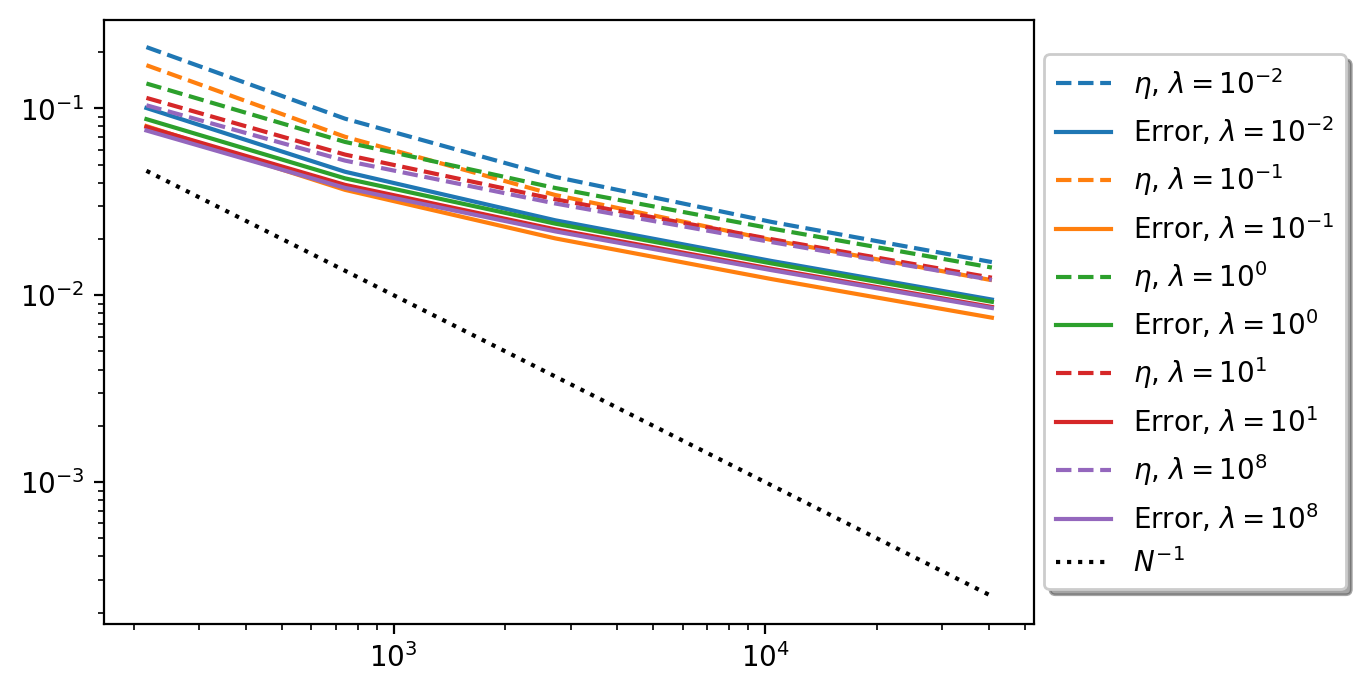}
    \caption{Error estimator and error for adaptive refinement, dörfler parameter 1}
    \label{fig:adaptiv-doerfler100}
\end{figure}

\end{document}